\documentclass[final,11pt]{siamart171218}
\usepackage{geometry}
\usepackage{amsfonts}
\usepackage{amssymb}
\usepackage{graphicx}
\usepackage{epstopdf}

\usepackage{algorithmic}
\Crefname{ALC@unique}{Line}{Lines} 
\ifpdf
  \DeclareGraphicsExtensions{.eps,.pdf,.png,.jpg}
\else
  \DeclareGraphicsExtensions{.eps}
\fi

\newsiamremark{hypothesis}{Hypothesis}
\crefname{hypothesis}{Hypothesis}{Hypotheses}
\newsiamthm{claim}{Claim}

\usepackage{amsopn}
\DeclareMathOperator{\diag}{diag}

\ifpdf
\hypersetup{
  pdftitle={Hybrid projection methods with recycling},
  pdfauthor={J. Chung, E. de Sturler, and J. Jiang}
}
\fi

\definecolor{blue}{rgb}{0,0,1}
\definecolor{red}{rgb}{1,0,0}
\definecolor{green}{rgb}{0,1,0}
\definecolor{orange}{rgb}{1,.5,0}

\newdimen\iwidth
\newdimen\iheight

\usepackage{tikz,pgflibraryplotmarks}
\usepackage{pgf,pgfplots,pgfarrows} 
\newcommand{\bs}[1]{\boldsymbol{#1}}

\include{macro_list}
\author{Julianne Chung \and Eric de Sturler \and Jiahua Jiang}

\headers{Hybrid projection methods with recycling}{J. Chung, E. de Sturler, J. Jiang}

\title{Hybrid Projection Methods with Recycling for Inverse Problems\thanks{Updated \today.
\funding{This work was funded by NSF DMS 1654175 and NSF DMS 1720305}}}

\author{Julianne Chung\thanks{Department of Mathematics, Computational Modeling and Data Analytics Division, Academy of Integrated Science, Virginia Tech, Blacksburg, VA
  (\email{jmchung@vt.edu}, \url{https://www.math.vt.edu/people/faculty/chung-julianne}).}
\and Eric de Sturler\thanks{Department of Mathematics, Computational Modeling and Data Analytics Division, Academy of Integrated Science, Virginia Tech, Blacksburg, VA
  (\email{sturler@vt.edu}, \url{https://www.math.vt.edu/people/faculty/desturler-eric}).}
  \and Jiahua Jiang\thanks{Department of Mathematics, Virginia Tech, Blacksburg, VA
    (\email{jiahua@vt.edu}, \url{https://intranet.math.vt.edu/people/jiahua}).}}

\begin{document}

\maketitle

\begin{abstract}
Iterative hybrid projection methods have proven to be very effective for solving large linear inverse problems due to their inherent regularizing properties as well as the added flexibility to select regularization parameters adaptively.
In this work, we develop Golub-Kahan-based hybrid projection methods that can exploit compression and recycling techniques in order to solve a broad class of inverse problems where memory requirements or high computational cost may otherwise be prohibitive. 
For problems that have many unknown parameters and require many iterations, hybrid projection methods with recycling can be used to compress and recycle the solution basis vectors to reduce the number of solution basis vectors that must be stored, while obtaining a solution accuracy that is comparable to that of standard methods. If reorthogonalization is required, this may also reduce computational cost substantially.  
In other scenarios, such as streaming data problems or inverse problems with multiple datasets, hybrid projection methods with recycling can be used to efficiently integrate previously computed information for faster and better reconstruction.

Additional benefits of the proposed methods are that various subspace selection and compression techniques can be incorporated, standard techniques for automatic regularization parameter selection can be used, and the methods can be applied multiple times in an iterative fashion. Theoretical results show that, under reasonable conditions, regularized solutions for our proposed recycling hybrid method remain close to regularized solutions for standard hybrid methods and reveal important connections among the resulting projection matrices. Numerical examples from image processing show the potential benefits of combining recycling with hybrid projection methods.
\end{abstract}

\begin{keywords}
Golub-Kahan bidiagonalization, hybrid projection methods, recycling, compression, inverse problems,
streaming data problems,
deconvolution, tomography, reduced basis methods
\end{keywords}

\thispagestyle{plain}

\section{Introduction}
\label{sec:introduction}
Inverse problems arise in many applications, where the goal is to approximate some unknown parameters of interest from indirect measurements or observations. For large-scale problems where the regularization parameter is not known in advance, iterative hybrid projection methods can be used to simultaneously estimate the regularization parameter and compute regularized solutions.
However, one of the main disadvantages of hybrid methods compared to standard iterative methods is the need to store the basis vectors for solution computation, which can present significant computational bottlenecks if many iterations are needed or if there are many unknowns. Furthermore, these methods are typically embedded within a larger problem that needs to be solved (e.g., optimal experimental design or nonlinear frameworks), so it may be required to solve a sequence of inverse problems (e.g., where the forward model is parameterized such that the change in the model from one problem to the next is relatively small) or to compute and update solutions from streaming data. Rather than start each solution computation from scratch, we assume that a few vectors for the solution subspace can be provided, and our goals are to improve upon the given subspace and to compute a regularized solution efficiently in the improved subspace.
In this paper, we develop \textit{recycling} Golub-Kahan-based hybrid projection methods that combine a recycling Golub-Kahan bidiagonalization (recycling GKB) process with tools from compression to solve a broad class of inverse problems.

We consider linear inverse problems of the form,
\begin{equation}
 \label{eq:linear problem}
 \bfb = \bfA \bfx_\true + \bfepsilon,
\end{equation}
where $\bfA \in \bbR^{M \times N}$
 models the forward process, $\bfb \in \bbR^M$ contains observed data, $\bfx_\true\in\bbR^N$ represents the desired parameters, and $\bfepsilon \in \bbR^M$ is noise or measurement error.  Given $\bfb$ and $\bfA$, the goal is to compute an approximation of $\bfx_\true$.  In this work, we are interested in solving the Tikhonov regularized problem,
\begin{equation}
	\label{eq:Tikhonov}
	\min_\bfx \norm[2]{\bfA \bfx - \bfb}^2 + \lambda^2 \norm[2]{\bfx}^2
\end{equation}
where $\lambda \geq 0$ is a (yet-to-be-determined) regularization parameter that balances the data-fit term and the regularization term.  We remark that extensions to the general-form Tikhonov problem can be made, which often requires a transformation to standard form \cite{hansen2010discrete}. Although the Tikhonov problem has been studied for many years, various computational challenges have motivated the development of hybrid iterative projection methods for computing an approximate solution to \cref{eq:Tikhonov}.  Basically, in a hybrid projection method, the original problem is projected onto small subspaces of increasing dimension and the projected problem is solved using variational regularization.  By regularizing the projected problem, hybrid methods can stabilize the convergence behavior of the method, and the regularization parameter does not need to be known in advance.
An additional benefit is that these iterative methods can handle problems where matrices $\bfA$ and $\bfA\t$ are so large that they can not be constructed but can be accessed via function evaluations.

In this paper, we propose hybrid projection methods that combine recycling techniques to improve a given solution subspace with an efficient approach to compute a regularized solution to the projected problem, with automatic regularization parameter selection. The general approach consists of three steps, which can be used in an iterative fashion. First, we begin with a suitable set of orthonormal basis vectors, denoted $\bfW_{k-1} \in \mathbb{R}^{N \times (k-1)}$.  This may be provided (e.g., from a related problem or from expert knowledge) or may need to be determined (e.g., via compression of previous solutions). With an initial guess of the solution, $\bfx^{(1)}$, the second step is to use a recycling GKB process to generate vectors that span a particular Krylov subspace, contained in $\wt{\bfV}_{\ell}\in \bbR^{N \times \ell}$, and extend the solution space to be $\calR\left(\begin{bmatrix} \bfW_{k-1} & \bfx^{(1)}&\wt{\bfV}_{\ell}\end{bmatrix} \right)$ where $\calR(\cdot)$ denotes the column space of a matrix.
The third step is to find a suitable regularization parameter $\lambda$ and compute a solution to the regularized projected problem,
  \begin{equation}
    \label{eq:regprojproblem}
    \min_{\bfx \in \calR\left(\begin{bmatrix} \bfW_{k-1} & \bfx^{(1)}& \wt{\bfV}_{\ell}\end{bmatrix}\right)} \norm[2]{\bfA\bfx - \bfb}^2 + \lambda^2 \norm[2]{\bfx}^2.
  \end{equation}
  The main approach (corresponding to steps 2 and 3) is described in \cref{sec:recyclingGKB} and \cref{sec:HyBRrecycle}, and some compression approaches that can be used in step 1 are provided in \cref{sec:compression_approaches}.

Recycling techniques for iterative methods have been considered for multiple Krylov solvers and a wide range of applications, but mainly  for square system matrices and for well-posed problem \cite{parks2006recycling,soodhalter2014krylov,Wang.TopOptRecyc.2007,KilmerdeSturler2006,Ahuja.Recyc-BiCGStab-MOD.2015,S.2016,jin2007parallel,keuchel2016combination,feng2009parametric,Recyc-Imp-Tom.2010}. Augmented LSQR methods have been described in \cite{baglama2013augmented,baglama2005augmented} for well-posed least-squares problems that require many LSQR iterations.  By augmenting Krylov subspaces using harmonic Ritz vectors that approximate singular vectors associated with the small singular values, this approach can reduce computational cost by using implicit restarts for improved convergence.  However, when applied to ill-posed inverse problems, the augmented LSQR method without an explicit regularization term exhibits semiconvergence behavior whereby the reconstructions eventually become contaminated with noise and errors.
Other approaches for augmenting or enriching Krylov subspaces are described in \cite{hansen2019hybrid,calvetti2003enriched,hochstenbach2010subspace}, where Krylov subspaces are combined with vectors containing important information about the desired solution (e.g., a low-dimensional subspace).   These methods can improve the solution accuracy by incorporating information about the desired solution into the solution process, but the improvement in accuracy \textit{significantly} depends on the quality of the provided vectors.  Modifications of conjugate gradient and TSVD are described in \cite{calvetti2003enriched, hochstenbach2010subspace}. A hybrid enriched bidiagonalization (HEB) method that stably and efficiently augments a ``well-chosen enrichment subspace'' with the standard Krylov basis associated with LSQR is described in \cite{hansen2019hybrid}.  Contrary to the HEB method, our recycling GKB method generates the extension subspace vectors $\wt{\bfV}_{\ell}$ such that we \emph{improve} on the space, rather than just augment it. Thus, as we will demonstrate in \cref{sec:numerics}, our approach can handle a wider range of problems and provide more accurate solutions.

The paper is organized as follows.  In \cref{sec:background}, we provide a brief overview on hybrid projection methods, where we focus on methods based on the standard Golub-Kahan bidiagonalization (GKB) process. Then in \cref{sec:recyclinghybrid}, we propose new hybrid projection methods that are based on the recycling GKB process and describe techniques for incorporating regularization automatically and efficiently. We also describe some examples of compression methods that can be used in Step 1 of the proposed approach and provide theoretical results.  In particular, we investigate the impact of compression and recycling on the projected problem and show important results that relate the regularized solution from a recycling approach to that from a standard approach. Numerical results are provided in \cref{sec:numerics}, and conclusions are provided in \cref{sec:conclusions}.

\section{Background on hybrid iterative methods}
\label{sec:background}
Hybrid approaches that embed regularization within iterative methods date back to seminal papers by O'Leary and Simmons in 1981 \cite{OlSi81} and Bjorck in 1988 \cite{Bjo88}, and the number of extensions and developments in the area of hybrid methods continues to grow.  We focus on hybrid methods based on the GKB process, which generate an $m$-dimensional Krylov subspace using
matrix $\bfA\t \bfA$ and vector $\bfA\t \bfb$,
$$\calK_m\left(\bfA\t \bfA,\bfA\t \bfb\right) = \mbox{span} \left\{\bfA\t \bfb, (\bfA\t \bfA) \bfA\t \bfb,
\dots , (\bfA\t \bfA)^{m-1} \bfA\t \bfb\right\}.$$  The GKB process\footnote{
We assume no termination of the iteration, and therefore the dimension of $K_{m}(\bfA\t \bfA,\bfA\t \bfb)$ is $m$.}
\cite{GoKa65} can be described as follows. Let $\beta_1 = \norm[2]{\bfb}$, $\bfu_1 = \bfb / \beta_1$, and $\alpha_1 \bfv_1 = \bfA\t \bfu_1$.
Then at the $j$-th iteration of the GKB process, we generate vectors $\bfu_{j+1}$ and $\bfv_{j+1}$ such that
\begin{equation}
\label{eq:gk1-its}
	\beta_{j+1}\bfu_{j+1} = \bfA \bfv_j - \alpha_j \bfu_j \quad \mbox{and} \quad
	\alpha_{j+1}\bfv_{j+1} = \bfA\t \bfu_{j+1} - \beta_{j+1} \bfv_j,
\end{equation}
and after $m$ iterations we have the relationships,
\begin{align}
	\label{eq:gk1}
	\bfA \bfV_m & = \bfU_{m+1} \mxB_{m}, \\
	\label{eq:gk2}
	\bfA\t \bfU_{m+1}&  = \bfV_m \mxB_m\t + \alpha_{m+1} \bfv_{m+1} \bfe_{m+1}\t = \bfV_{m+1} \bfL_{m+1}\t,
\end{align}
where $\bfV_{m} = \begin{bmatrix} \bfv_1 & \dots & \bfv_m \end{bmatrix}  \in \bbR^{N \times m}$ and $\bfU_{m+1} = \begin{bmatrix} \bfu_1 & \dots & \bfu_{m+1}\end{bmatrix} \in \bbR^{M \times (m+1)}$ contain orthonormal columns, bidiagonal matrix
\begin{equation}
\label{def:Bm}
	\mxB_{m} =\begin{bmatrix}
		\alpha_1 & & & \\
		\beta_2 & \alpha_2& & \\
		 & \ddots& \ddots & \\
		 & & \beta_m & \alpha_m \\
		 & & & \beta_{m+1}  \\
	\end{bmatrix}\in \bbR^{(m+1) \times m} ,
\end{equation} and $\bfL_{m+1} = \begin{bmatrix} \bfB_m & \alpha_{m+1} \bfe_{m+1}
\end{bmatrix}$.  Given these relations, an approximate least-squares solution can be computed as $\bfx_m = \bfV_m\bfy_m$ where $\bfy_m$ is the solution to the \emph{projected} least-squares problem,
\begin{equation}
\label{eq:PLS}
  \min_{\bfx \in \calR(\bfV_m)} \|\bfA\bfx - \bfb\|_2^2 =
  \min_{\bfy} \|\mxB_m \bfy - \beta_1\bfe_1\|_2^2.
\end{equation}
In standard LSQR implementations, the columns of $\bfV_m$ and $\bfU_{m+1}$ do not need to be stored and efficient updates can be used to minimize storage requirements. For iterative methods, the main computational cost at each iteration is a matrix-vector product with $\bfA$ and its transpose.  The storage cost for these iterative methods is very low (e.g., $M+2N$ for LSQR) due to a 3-term recurrence property.

However, when applied to ill-posed inverse problems, standard iterative methods exhibit semi-convergent behavior, whereby solutions improve in early iterations but become contaminated with inverted noise in later iterations \cite{hansen2010discrete}.  Thus, it is desirable to consider a hybrid iterative projection method that combines iterative regularization with a variational regularization method such as Tikhonov regularization.  One approach is to solve the Tikhonov problem \cref{eq:Tikhonov} by applying any iterative least-squares solver (e.g., LSQR) to the equivalent augmented system,
\begin{equation}
	\label{eqn:tikhonov3}
	\min_\bfx \norm[2]{\begin{bmatrix} \bfA \\ \lambda \bfI	\end{bmatrix} \bfx - \begin{bmatrix}\bfb \\ \bfzero
		\end{bmatrix}}^2\,.
\end{equation}
The main challenge is that the regularization parameter $\lambda$ must be selected a priori, which can be difficult especially for large-scale problems. Another hybrid iterative approach is to project the problem onto Krylov subspaces of increasing dimension and to compute the solution at the $m$-th iteration as $\bfx_m = \bfV_m\bfy_m$ where $\bfy_m$
solves the \emph{projected, regularized} problem,
\begin{equation}
	\label{eq:projreg}
  \min_{\bfx \in \calR(\bfV_m)} \|\bfA\bfx - \bfb\|_2^2 + \lambda^2 \norm[2]{\bfx}^2 =
  \min_{\bfy} \|\bfB_{m}\bfy - \beta_1\bfe_1\|_2^2 + \lambda^2 \norm[2]{\bfy}^2.
\end{equation}
One benefit of this approach is that the regularization parameter for the projected problem can be easily and automatically estimated during the iterative process \cite{KiOLe01,ChNaOLe08,renaut2010regularization}. However, a potential disadvantage is the storage of $\bfV_m$ which is needed for solution computation.
For some problems where the solution can be represented in only a few basis vectors, this additional storage is not a concern.  However, for large-scale problems where storage of these vectors becomes too demanding, the proposed hybrid projection methods with recycling and compression that we describe in the next section can be used to reduce this computational cost.

\section{Hybrid projection methods with recycling}
\label{sec:recyclinghybrid}
Using iterative hybrid projection methods to solve large-scale inverse problems can be quite effective. We are interested  in scenarios where one has an initial solution subspace (e.g., from a prior reconstruction or from a sequence of reconstructions), and the goal is to incorporate such information to not only augment but also improve or enhance the solution subspace, thereby improving the quality of the subsequent solution approximations. For example, for problems requiring many iterations, the memory and storage costs required to store the basis vectors for solution computation in canonical hybrid projection methods can exceed capabilities or result in significantly longer compute times. The proposed {\em hybrid projection methods with recycling} can be used to ameliorate the memory requirements without sacrificing the quality of the solution, where a main ingredient is the recycling GKB process. Here, we modify the classical GKB process to augment and enhance a given orthonormal basis.  Then, the recycling GKB process can be combined with a regularization technique to give an efficient hybrid projection method. Finally, by exploiting various compression approaches, compression and recycling can be repeated in an iterative fashion until a desired reconstruction is obtained.
An overview of the general approach is provided in Algorithm \ref{alg:hybridrecycle}.

\begin{algorithm}
\caption{Hybrid projection method with recycling and compression}
\label{alg:hybridrecycle}
\begin{algorithmic}[1]
	\REQUIRE{$\bfA$, $\bfb$, $\bfW_{k-1}$, $\bfx^{(1)}$
	}
\WHILE{desired solution not obtained}
\STATE $\ell = 1$
\STATE Construct $\bfW_{k}$; see \cref{sec:recyclingGKB}.
\WHILE{storage is available and stopping criteria not satisfied}
  \STATE Use recycling GKB to compute augmented subspace $\widetilde \bfV_\ell$; see \cref{sec:recyclingGKB}.
  \STATE Compute regularization parameter.
  \STATE Solve regularized, projected problem; see \cref{sec:HyBRrecycle}.
  \STATE $\ell = \ell+1$
\ENDWHILE
\STATE Use compression to get $\bfW_{k-1}$; see \cref{sec:compression_approaches}.
\ENDWHILE
\end{algorithmic}
\end{algorithm}

Notice that even though a large number of iterations can be performed, the size of the projected problem will never exceed the set storage limit.  Furthermore, theoretical results provided in \cref{subsec:analysis} show that under reasonable conditions, regularized solutions obtained from the recycling GKB approach remain close to the standard GKB solution. 
We will also address a special case where $\bfW_{k-1}$ and $\bfx^{(1)}$ come from a standard Krylov approach and TSVD is used for compression.

For all derivations and results in this section, we assume exact arithmetic and no breakdown of the algorithms. 

\subsection{Recycling Golub-Kahan bidiagonalization}
\label{sec:recyclingGKB}

In this section, we assume that an approximate solution (or initial guess) $\bfx^{(1)}$ and a matrix $\bfW_{k-1} \in \bbR^{N \times (k-1)}$ with orthonormal columns are given, and we describe the recycling GKB process that can be used to augment the solution subspace using recycling techniques. First, assuming $\bfx^{(1)}\notin \calR(\bfW_{k-1})$, we
set $\bfW_k  = \begin{bmatrix} \bfW_{k-1} & \widecheck{\bfx}^{(1)} \end{bmatrix} \in \bbR^{N \times k}$ where $\widecheck{\bfx}^{(1)} = \left(\bfx^{(1)} - \bfW_{k-1}\bfW_{k-1}\t \bfx^{(1)}\right)/ \norm[2]{\bfx^{(1)} - \bfW_{k-1}\bfW_{k-1}\t \bfx^{(1)}}$.
Now, $\bfW_k$ represents the 
\emph{recycled} subspace and $\bfW_k\t\bfW_k = \bfI_k$, and the approximate solution (or initial guess) $\bfx^{(1)}$ is always in the search space. Thus subsequent (regularized) approximations may preserve this search direction.  If $\bfx^{(1)} \in \calR(\bfW_{k-1})$, then $\bfW_{k-1}$ can be used as the recycled subspace.

Next, take the skinny QR factorization of $\bfA \bfW_k$,
\begin{equation}
\label{AWYR}
\bfA \bfW_k = \bfY_k \bfR_k \in \bbR^{M \times k} ,
\end{equation}

compute $\widecheck{\bfr}^{(1)} = \bfb - \mxA \widecheck{\bfx}^{(1)}$,
and set
\eqs
  \wt{\bfb} & = & \widecheck{\bfr}^{(1)}  -  \bfY_k \bfzeta  \quad \text{where} \quad \bfzeta =  \bfY_k\t \widecheck{\bfr}^{(1)} .
\eqe
The basic approach is to extend the solution space with an additional $\ell$ vectors generated by the recycling GKB process. Starting with
$\wt{\bfu}_1 = \wt{\bfb}/\wt{\beta}_1$ where $\wt{\beta}_1 = \| \wt{\bfb} \|_2 $ (note $\wt{\bfu}_1 \perp \bfY_k$) and $\wt{\a}_1\wt{\bfv}_1  =  \mxA\t \wt{\bfu}_1$,
at the $j$-th iteration of the recycling GKB process, we generate vectors $\wt{\bfu}_{j+1}$ and $\wt{\bfv}_{j+1}$ as
\eqs
\label{eq:recur1}
  \wt{\b}_{j+1}\wt{\bfu}_{j+1} & = & \left(\mxI - \bfY_k \bfY_k\t\right)\mxA \wt{\bfv}_j - \wt{\a}_j \wt{\bfu}_j ,\\
  \label{eq:recur}
  \wt{\a}_{j+1} \wt{\bfv}_{j+1} & = &  \mxA\t \wt{\bfu}_{j+1} - \wt{\b}_{j+1} \wt{\bfv}_{j},
\eqe
and after $\ell$ iterations, we have the following recurrence relation, cf.~(\ref{eq:gk1}) -- (\ref{eq:gk2}),
\begin{align}
\label{AVUB1}
  \left(\mxI - \bfY_k \bfY_k\t\right)\mxA \wt{\bfV}_\ell   & =  \wt{\bfU}_{\ell+1} \wt{\mxB}_{\ell}   \\
  \mxA\t \wt{\bfU}_{\ell+1} &  =  \wt{\bfV}_{\ell} \wt{\mxB}_{\ell}\t + \wt{\a}_{\ell+1} \wt{\bfv}_{\ell+1} \bfe_{\ell+1}\t, \label{AVUB2}
\end{align}
where $\wt{\bfV}_\ell = \begin{bmatrix} \wt{\bfv}_1 & \dots & \wt{\bfv}_\ell \end{bmatrix} \in \bbR^{N \times \ell},$ $\wt{\bfU}_{\ell+1} = \begin{bmatrix} \wt{\bfu}_1 & \dots & \wt{\bfu}_{\ell+1}\end{bmatrix} \in \bbR^{M \times (\ell+1)}$, and bidiagonal matrix $\wt{\bfB}_\ell \in \bbR^{(\ell+1) \times \ell}$
is constructed during the iterative process.
Notice that by construction $\wt{\bfU}_{\ell+1}\t \bfY_k = \mxO$, 
where $\mxO$ is the zero matrix, and hence
$\bfA\t \wt{\bfu}_j \perp \bfW_k$ for $j = 1, \ldots, (\ell+1)$, since $\bfW_k\t \bfA\t \wt{\bfU}_{\ell+1} = \bfR_k\t \bfY_k\t\wt{\bfU}_{\ell+1} = \mxO$. Hence,  $\wt{\bfv}_1 \perp \bfW_k$, and if we assume that $\wt{\bfv}_i \perp \bfW_k$ for $i = 1, \ldots, j$, then by induction we have from \eqref{eq:recur},
$$\wt{\alpha}_{j+1}\bfW_k\t \wt{\bfv}_{j+1} = \bfW_k\t \bfA\t 
\wt{\bfu}_{j + 1} - \wt{\b}_{j+1} \bfW_k\t \wt{\bfv}_{j}= \bf0.$$
Thus, $\bfW_k \perp \wt{\bfV}_\ell$ in exact arithmetic, without explicit orthogonalization.
We notice from (\ref{AVUB1}) that $\mxA \wt{\bfV}_\ell  =  \bfY_k \bfY_k\t\mxA \wt{\bfV}_\ell + \wt{\bfU}_{\ell+1} \wt{\mxB}_{\ell}$, so we have the recycling GKB relation,
\begin{equation}
\label{eq:new_gkb}
  \bfA \begin{bmatrix} \bfW_k & \wt{\bfV}_\ell \end{bmatrix}   = \begin{bmatrix} \bfY_k & \wt{\bfU}_{\ell+1}\end{bmatrix}
  \begin{bmatrix} \bfR_k & \bfY_k\t\bfA \wt{\bfV}_\ell \\ \bf0 & \wt{\bfB}_{\ell}\end{bmatrix} ,
\end{equation}
where $\begin{bmatrix} \bfW_k &  \wt{\bfV}_\ell \end{bmatrix}$  and $\begin{bmatrix} \bfY_k & \wt{\bfU}_{\ell+1}\end{bmatrix}$
both contain orthonormal columns.

Thus far, we have described a recycling GKB approach that can be used to augment a given solution subspace.  A distinguishing factor of this approach compared with existing enhancement methods is that the 
new, augmented, Krylov subspace depends on the recycled subspace.  Indeed, one can characterize the augmented solution subspace as a Krylov subspace of the form
$$\calR\left(\wt{\bfV}_\ell\right)=\calK_\ell\left(\bfA\t \left(\bfI - \bfY_k \bfY_k\t\right)\bfA, \bfA\t \left(\bfI - \bfY_k \bfY_k\t\right)\bfr^{(1)}\right).$$

\subsection{Hybrid projection methods using the recycling GKB}
\label{sec:HyBRrecycle}

Next, we describe how the recycling GKB process can be incorporated within a hybrid projection method for efficient regularized solution computation. 
Suppose we have performed $\ell$ iterations of the recycling GKB process, and we are interested in computing approximate Tikhonov solutions in the augmented solution subspace $\calR\left( \ars{cc}\bfW_{k} & \wt{\bfV}_{\ell}\are\right)$, i.e., we are looking for solutions of the form $\bfx_{k,\ell} = \ars{cc}\bfW_{k} & \wt{\bfV}_{\ell}\are 
\bfy$ where
$\bfy=\begin{bmatrix}
\bfc \\ \bfd\end{bmatrix}$ for some vectors $\bfc\in\bbR^k$ and $\bfd\in\bbR^\ell$.
Using the fact that $\wt{\bfb} = \widecheck{\bfr}^{(1)} - \bfY_k \bfzeta = \wt{\beta}_1\wt{\bfU}_{\ell+1}{\bfe}_1$ and 
$\widecheck{\bfx}^{(1)}=\bfW_k \bfe_k$, we have
\begin{align}
\bfb &= \widecheck{\bfr}^{(1)} + \bfA\widecheck{\bfx}^{(1)} = {\bfY}_k{\bf\zeta}  + \wt{\beta}_1\wt{\bfU}_{\ell+1}{\bfe}_1 + \bfA{\bfW}_k{\bfe}_k\\
     &= \ars{cc}{\bfY}_k & \wt{\bfU}_{\ell+1}\are \ars{c}{\bf\zeta} + {\bfR}_k {\bfe}_k \\ \wt{\beta}_1{\bfe}_1\are.
\end{align}
Then, using \eqref{eq:new_gkb}, the residual can be written as
\eqs
\bfb - \mxA \begin{bmatrix} \bfW_k & \wt{\bfV}_\ell \end{bmatrix} \begin{bmatrix} \bfc \\ \bfd\end{bmatrix} = \begin{bmatrix} \bfY_k & \wt{\bfU}_{\ell+1}\end{bmatrix}
  \left( \ars{c} \bfzeta + \mxR_k\bfe_{k} \\ \wt{\beta}_1 \bfe_1 \are -
    \ars{cc} \mxR_k & \bfY_k\t\mxA\wt{\bfV}_\ell \\ \mxO & \wt{\mxB}_{\ell} \are
    \ars{c} \bfc \\ \bfd \are
  \right).
\eqe
Thus, the next iterate of the hybrid projection method with recycling is given by
\begin{equation}
\label{eq:x_recycle}
  \bfx^{(2)}  =  \ars{cc} \bfW_k & \wt{\bfV}_\ell \are \wt{\bfy}_\lambda ,
\end{equation}
where
\begin{equation}
\label{eq:projectedreg}
\wt{\bfy}_\lambda = \argmin_{\bfy} \norm[2]{
    \begin{bmatrix} \mxR_k & \bfY_k\t\mxA\wt{\bfV}_\ell \\ \mxO & \wt{\mxB}_{\ell} \end{bmatrix} \bfy
- \begin{bmatrix} \bfzeta + \mxR_k\bfe_{k} \\ \wt{\beta}_1 \bfe_1 \end{bmatrix} }^2  + \lambda^2\norm[2]{\bfy}^2.
\end{equation}

Notice that the coefficient matrix in the projected problem,
\begin{equation}
\label{bhat}
\widehat{\bfB}_{k,\ell} = \begin{bmatrix} \bfR_k & \bfY_k\t\bfA \wt{\bfV}_\ell \\ \bf0 & \wt{\bfB}_{\ell}
\end{bmatrix} \in \mathbb{R}^{(k+\ell+1) \times (k+\ell)}
\end{equation}
is modest in size.  Thus, standard regularization parameter selection methods can be used to choose $\lambda$.
Based on the above derivation, we can interpret iterates of the hybrid projection method with recycling as optimal solutions in a $(k+\ell)$ dimensional subspace.  That is, for fixed $\lambda \geq 0$,
\begin{equation}
\bfx^{(2)} = \argmin_{\bfx \in \calR\left( \ars{cc}\bfW_{k} & \wt{\bfV}_{\ell}\are\right)} ||\bfA\bfx - \bfb ||^{2}_{2} +\lambda^2 \norm[2]{\bfx}^2. 
 \end{equation}

If the solution is not sufficiently accurate, the process can be repeated in an iterative fashion by selecting a new subspace $\calR\left(\bfW_{k-1}^{(new)}\right) \subset \calR\left(\ars{cc} \bfW_k & \wt{\bfV}_\ell \are\right)$ (for example, using one of the compression approaches in the next section), where $\bfW_{k-1}^{(new)}$ has orthonormal columns, and set $\bfW_k^{(new)} = \begin{bmatrix}\bfW_{k-1}^{(new)} & \widecheck{\bfx}^{(2)}\end{bmatrix}$ with 
$$\widecheck{\bfx}^{(2)} = \left(\left(\bfI - \bfW_{k-1}^{(new)}\left(\bfW_{k-1}^{(new)}\right)\t\right)\bfx^{(2)}\right) / \left\| \left(\bfI - \bfW_{k-1}^{(new)}\left(\bfW_{k-1}^{(new)}\right)\t\right)\bfx^{(2)} \right\|_2.$$ 
Note, that $\calR\left( \bfW^{(new)}_k \right) \subset \calR \left(\begin{bmatrix} \bfW_k & \wt{\bfV}_\ell \end{bmatrix}\right)$.
Next, we set 
$\widecheck{\bfr}^{(2)}  =  \bfb - \mxA \widecheck{\bfx}^{(2)}$, and repeat the steps above.

We remark on the additional computational cost if full reorthogonalization is desired.  In particular, the recursion,
\begin{equation}
\wt{\a}_{j+1} \wt{\bfv}_{j+1}  = \left(\bfI - \bfW_k \bfW_k\t\right) \mxA\t \wt{\bfu}_{j+1} - \wt{\b}_{j+1}\wt{\bfv}_{j}.
\end{equation}
can be used in place of (\ref{eq:recur}) to ensure that the solution basis vectors $\begin{bmatrix}\bfW_{k} & \wt{\bfV}_{\ell}\end{bmatrix}$ are orthogonal in floating point arithmetic.
In this case, the additional computational cost is $4kN$ operations for each iteration.

\subsection{Compression approaches}
\label{sec:compression_approaches}
One feature of the hybrid projection methods with recycling is the ability to combine compression and extension of the solution space in an iterative manner. That is, compression techniques can reduce the total number of solution vectors that we need to store, which can be followed by  
enhancement of the space, and this can be done without significantly degrading the accuracy of the resulting reconstruction. 
More specifically, let $\bfV^c$ represent the current set of basis vectors, and assume that we can only afford to store $m$ vectors of length $N$. When the number of columns in $\bfV^c$ reaches $m$, we can compress the vectors in $\bfV^c$ to get $\bfW_{k-1} \in \mathbb{R}^{N \times (k-1)}$ (see line 9 in \cref{alg:hybridrecycle}).  Then, we can construct $\bfW_k$ using an initial guess or current approximate solution and use the method described in \cref{sec:recyclingGKB} to augment the space with $\wt{\bfV}_{\ell}$, where $\ell = m-k$.

In this section, we focus on four compression strategies for constructing $\bfW_{k-1}$ that are well-suited for solving inverse problems with the recycling GKB process.
These include truncated singular value decomposition (TSVD), solution-oriented compression, sparsity enforcing compression, and reduced basis decomposition (RBD).
The described compression strategies follow two perspectives: (1) decompose $\widehat{\bfB}_{k,\ell} \in \bbR^{(m+1)\times m}$ defined in \cref{bhat} and use truncation (e.g., TSVD and RBD) or (2) use components in the solution of the projected problem (\ref{eq:projectedreg}) to identify the important columns of $\bfV^c$ (e.g., sparsity enforcing and solution-oriented compression). 
Throughout this subsection, we define $1 \le q < m$ as the largest number of length $N$ vectors we wish to keep after compression and $\epsilon_{tol} > 0$ is a tolerance for the compression.

First we describe the TSVD approach for compressing $\bfV^c$.
Let the SVD of $\widehat{\bfB}_{k,\ell}$ 
be given as
\begin{equation}
\label{eq:svdBm}
\widehat{\bfB}_{k,\ell} = {\bf\Psi}_{m+1}{\bf\Sigma}_m{\bf\Phi}_m\t,
\end{equation}
where ${\bf\Psi}_{m+1} \in \bbR^{(m+1) \times (m+1)} $ and ${\bf\Phi}_m \in \bbR^{m \times m}$ are orthogonal matrices, and  $\bs{\Sigma}_{m} \in \bbR^{(m+1) \times m}$ is a diagonal matrix containing singular values $\sigma_i, i = 1, \ldots, m$.
If $\sigma_{q} < \epsilon_{tol}$, we let $k-1 = i$, where $i$ is the largest index such that $\sigma_{i}  \geq \epsilon_{tol}$, otherwise $k-1=q$.
The key point of this compression strategy is that we identify the important columns of $\bfV^c$ as those corresponding to the large singular values of $\widehat{\bfB}_{k,\ell}$.
 The compressed representation of $\bfV^c$ is given by
\begin{equation}
\label{tsvd_w}
 \bfW_{k-1} = \bfV^c{\bf\Phi}_{k-1},
\end{equation}
where ${\bf\Phi}_{k-1}$ contains the first $k-1$ columns of ${\bf\Phi}_m$.

The second compression approach is motivated by the notion that the absolute value of each component of the solution to the projected problem $\wt{\bfy}_\lambda = \begin{bmatrix}\tilde{y}_{1}, &\dots, &\tilde{y}_{m}\end{bmatrix}\t$ is indicative of the important columns of $\bfV^c$.
We define $I_{m}, J_m$ as an index set at the $m$-th iteration:
\begin{align}
I_{m} &= \{i:  \quad |\tilde{y}_{i}| > \epsilon_{tol},  \;  1 \le i \le N \}\\
J_{m} &= \{i:  \quad |\tilde{y}_{i}| \; \text{are the largest} \; q \; \text{components}, \;  1 \le i \le N \}
\end{align}
For solution-oriented compression, we define $\bfW_{k-1}  = \begin{bmatrix} \bfv_{m_{1}} & \dots & \bfv_{m_{k-1}}\end{bmatrix}$, where $k - 1 = |I_{m} \cap J_m|$, $\{m_{j}\}^{k-1}_{j=1} \subseteq (I_{m} \cap J_m)$ and $m_1 \le \dots \le m_{k-1}$.

The third compression approach called sparsity-enforcing compression is intuitively similar to
the solution-oriented method. The basic idea is to use $\wt{\bfy}_\lambda$ to identify the important vectors in $\bfV^c$; however, the difference is that we employ a sparsity enforcing regularization term on the projected problem.  A standard algorithm such as SpaRSA \cite{wright2009sparse} can be used to solve for $\wt{\bfy}_\lambda$, and then corresponding vectors of $\bfV^c$ can be extracted similar to solution-oriented compression.
Lastly, we exploit tools from reduced order modeling \cite{chen2015reduced} to compress the solution vectors.  For $1 \le i \le q$, we consider the reduced basis decomposition of $\widehat{\bfB}_{k,\ell}\t$, 
\begin{equation}
\widehat{\bfB}_{k,\ell}\t = \bfS_{i}\bfT_{i},
\end{equation}
where $\bfS_{i} \in \mathbb{R}^{m \times i}$ contains orthonormal columns and transformation matrix $\bfT_{i} \in \mathbb{R}^{i \times (m+1)}$. Define $\mathcal{E}_i = \underset{1\le j \le (m+1)}{\max}\norm[2]{\widehat{\bfB}_{k,\ell}\t(:,j) - \bfS_{i}\bfT_{i}(:,j)}$. If $\mathcal{E}_q < \epsilon_{tol}$, we let $k-1 = i$, where $i$ is the largest index such that $\mathcal{E}_i \leq \epsilon_{tol}$, otherwise $k-1=q$. 
We use $\bfS_{k-1}$ to indicate important columns of $\bfV^c$, thus the compressed vectors are obtained as $\bfW_{k-1} = \bfV^c\bfS_{k-1}$.

\subsection{Theoretical analysis of hybrid projection methods with recycling}
\label{subsec:analysis}

In this section, we analyze theoretical properties of regularized solutions and the projected system using compression and recycling, in the important case that we run $m$ steps of standard GKB (see section \ref{sec:background}), compress the search space to dimension $k$, as described in section \ref{sec:compression_approaches} with ${\cal R}(\bfW_k) \subset {\cal R}(\bfV_m)$, and carry out $\ell$ steps of recycling GKB (see section \ref{sec:recyclingGKB}) which is incorporated in a hybrid projection method (see section \ref{sec:HyBRrecycle}).   
This scenario corresponds to the case where we can 
store a maximum of $m$ vectors of length $N$, but a hybrid projection method with standard GKB requires more iterations to converge.

First, we analyze the storage requirements. Let $j$ denote the number of iterations for a standard hybrid method. Without full reorthogonalization, we need to save $\bfV_{j} \in \mathbb{R}^{N \times j}$, bidiagonal matrix $\bfB_{j} \in \mathbb{R}^{(j+1)\times j}$, and $\bfu_{j+1} \in \mathbb{R}^{M \times 1}$, where the storage cost is dominated by $\bfV_{j}$ if $N$ is large. The total storage cost of standard hybrid iterative methods is
\begin{equation*}
\label{cost:GKB}
\mathcal{C}_{\text{HyBR}}(j) := 2j + (N+2)j + M.
\end{equation*}
As $j$ increases, $\mathcal{C}_{\text{HyBR}}(j)$ is dominated by $Nj$. Thus, for very large-scale problems, $\mathcal{C}_{\text{HyBR}}(j)$ increases rapidly and can easily exceed the storage limit. For the proposed recycling GKB hybrid method, we need to save $\bfW_{k} \in \mathbb{R}^{N \times k}, \bfY_{k}
\in \mathbb{R}^{M \times k}, \bfR_{k} \in \mathbb{R}^{k \times k}, \bfe_{k} \in \mathbb{R}^{k \times 1}, \wt{\bfB}_{\ell} \in \mathbb{R}^{(\ell +1) \times \ell}, \wt{\bfV}_{\ell} \in \mathbb{R}^{N \times \ell}, \bfzeta \in \mathbb{R}^{k \times 1}$, $\bfY_k^\top \bfA \wt{\bfV}_{\ell} \in \mathbb{R}^{k \times \ell}$ and $\wt{\bfu}_{\ell} \in \mathbb{R}^{M \times 1}$, where $\bfR_k$ is an upper triangular matrix and $\wt{\bfB}_{\ell}$ is a bidiagonal matrix. Since $\ell = m - k$, the storage cost of recycling GKB is
\begin{align*}
\label{cost:recyc-GKB}
\mathcal{C}_{\text{HyBR-recycle}} & := k^2/2 + (N+M+2)k + 2\ell + (N+1)\ell + k\ell \\
&=(N+2)m + Mk + k^2/2 + \ell(k+1)\\
                                  & < m^2/2 + (N+M+2)m.
\end{align*}
Therefore, the storage requirements of $\mathcal{C}_{\text{HyBR-recycle}}$ do not grow with the number of iterations.

Next, we consider several consequences of compression and augmentation
for the projected problem.
We are interested in comparing the properties of the GKB matrix $\widehat{\bfB}_{k,\ell}$ obtained with recycling with properties of the GKB matrix $\bfB_{m+\ell}$ obtained with $m+\ell$ standard GKB iterations. In addition, we show that under reasonable assumptions and for the same regularization parameter, the regularized solution from the recycling approach is close to the regularized solution from the standard approach.
For the particular case of TSVD compression, we give precise and (a posteriori) computable bounds.

We start with a lemma that shows important relations between the generated subspaces and then consider its consequence for relations between $\widehat\bfB_{k,\ell}$ and $\bfB_{m+\ell}$.
\begin{lemma}
\label{lem3}
Let $\bfV_{m+\ell}$, $\bfU_{m+\ell+1}$, and $\bfB_{m+\ell}$ be the matrices computed after $m+\ell$ iterations of standard GKB, following \eqref{eq:gk1-its}.
Let $\bfx^{(1)}$ be a (arbitrary) regularized solution computed from $\calR(\bfV_{m})$, $\calR(\bfW_{k-1}) \subset \calR(\bfV_m)$ (obtained by any compression method), and $\bfW_k$ be computed as described at the start of Section \ref{sec:recyclingGKB} with $\bfY_k$, $\bfR_k$ given in \eqref{AWYR}.
In addition, let
$\wt{\bfU}_{\ell+1} = \begin{bmatrix} \wt{\bfu}_{1} & \dots & \wt{\bfu}_{\ell+1}\end{bmatrix}$ and $\wt{\bfV}_{\ell} = \begin{bmatrix} \wt{\bfv}_{1}& \dots & \wt{\bfv}_{\ell}\end{bmatrix}$ be obtained after $\ell$ iterations of recycling GKB following \eqref{eq:recur1}--\eqref{eq:recur}.
Then
\begin{equation}
\label{eq3}
  \calR\left(\wt{\bfU}_{\ell+1}\right) \subset \calR(\bfU_{m+\ell+1}) \quad {\text and } \quad \calR\left(\wt{\bfV}_{\ell}\right) \subset \calR(\bfV_{m+\ell}).
\end{equation}
\end{lemma}
\begin{proof}
We prove the result by induction.
In recycling GKB, $\wt{\bfu}_{1} = \wt{\bfb}/{{\|\wt{\bfb}\|}_{2}}$ with $\wt{\bfb} = \widecheck{\bfr}^{(1)}  -  \bfY_k \bfY_k\t 
\widecheck{\bfr}^{(1)}$ and 
$\widecheck{\bfr}^{(1)} = \bfb - \mxA \widecheck{\bfx}^{(1)}$.
By construction, $\widecheck{\bfx}^{(1)}\in \calR(\bfV_m)$, and hence $\mxA \widecheck{\bfx}^{(1)} \in \calR(\mxA\bfV_{m})$. We also have $\calR(\bfY_k) \subset \calR(\mxA\bfV_{m})$, and, using \eqref{eq:gk1}--\eqref{eq:gk2}, $\calR(\mxA\bfV_{m}) \subset \calR(\bfU_{m+1})$. Since $\bfb = \beta_1 \bfu_1 \in \calR(\bfU_{m+1})$, we have
$\wt{\bfu}_{1} \in \calR(\bfU_{m+1})$.
Therefore, $\bfA\t \wt{\bfu}_{1} \in \calR(\bfA\t \bfU_{m+1}) = \calR( \bfV_{m+1})$, and as $\wt{\a}_1{\wt{\bfv}}_1  =  \mxA\t {\wt{\bfu}}_1$,
${\wt{\bfv}}_1 \in \calR( \bfV_{m+1})$.
Since $\bfA\wt{\bfv}_{1} \in \calR(\bfA\bfV_{m+1}) \subset \calR(\bfU_{m+2})$ and $\wt{\b}_2\wt{\bfu}_2 =  (\mxI - \bfY_k \bfY_k\t)\mxA \wt{\bfv}_1 - \wt{\a}_1\wt{\bfu}_1$, it follows that
$\wt{\bfu}_2  \in \calR(\bfU_{m+2})$.
Now assume that ${\wt \bfu}_{i+1} \in \calR( \bfU_{m+i+1})$ and ${\wt \bfv}_{i} \in \calR( \bfV_{m+i})$ for $i = 1, \ldots, j$.
Since $\calR(\bfY_k) \subset \calR(\bfU_{m+1})$,  we get from \eqref{eq:recur1}--\eqref{eq:recur} that
\begin{equation*}
  \wt\bfv_{j+1} \in \calR( \bfV_{m+j+1}) \quad\mbox{ and }\quad \wt{\bfu}_{j+1} \in \calR(\bfU_{m+j+1}).
\end{equation*}
\mbox{}
\end{proof}
The next result is presented without its (straightforward) proof. 
\begin{lemma}
Let $\mxA \in \bbR^{M \times N}$ and $\bfb \in \bbR^{M}$, and let $\mxP \in \bbR^{N \times N}$ and
$\mxQ \in \bbR^{M \times M}$ be orthogonal matrices. 
For any given $\l$, the Tikhonov solutions,
\eqsn
  \bfx_{\l} & = & \argmin_{\bfx \in \bbR^{N}} \| \mxA \bfx - \bfb\|_2^2 +
      \l^2 \| \bfx \|_2^2 \\
  \wt{\bfx}_{\l} & = & \argmin_{\wt{\bfx} \in \bbR^{N}}
      \| \mxQ \mxA \mxP\t \wt{\bfx} - \mxQ \bfb\|_2^2 +
      \l^2 \| \wt{\bfx} \|_2^2 
\eqen
satisfy $\bfx_{\l}  = \mxP\t \wt{\bfx}_{\l}$.
\end{lemma}

Next, we derive the orthogonal transformations that relate the Lanczos bases for the recycling GKB iteration with compression to those of the standard GKB iteration, and
the resulting relations between $\wh{\mxB}_{k,\ell}$
and $\mxB_{m+\ell}$. From Lemma~\ref{lem3} and the construction of
$\mxW_k$ and $\mxY_k$, we see that
$\calR(\mxW_k) + \calR(\wt{\mxV}_{\ell}) \subset \calR(\mxV_{m+\ell})$
and
$\calR(\mxY_k) + \calR(\wt{\mxU}_{\ell+1}) \subset \calR(\mxU_{m+\ell+1})$.
In addition, by construction the matrices $\begin{bmatrix} \mxW_k & \wt{\mxV}_{\ell}\end{bmatrix}$,
$\begin{bmatrix} \mxY_k & \wt{\mxU}_{\ell+1}\end{bmatrix}$, $\mxV_{m+\ell}$, and $\mxU_{m+\ell+1}$
have orthonormal columns. Hence, there exist orthogonal matrices
$\mxT = \begin{bmatrix} \mxT_1 & \mxT_2 & \mxT_c \end{bmatrix} \in \bbR^{(m+\ell+1) \times (m+\ell+1)}$
and
$\mxZ = \begin{bmatrix} \mxZ_1 & \mxZ_2 & \mxZ_c \end{bmatrix} \in \bbR^{(m+\ell) \times (m+\ell)}$, such that
$\mxY_k = \mxU_{m+\ell+1}\mxT_1$, $\wt{\mxU}_{\ell+1} = \mxU_{m+\ell+1}\mxT_2$, $\mxW_k = \mxV_{m+\ell}\mxZ_1$, and
$\wt{\mxV}_{\ell} = \mxV_{m+\ell}\mxZ_2$.
The subspace $\calR( \mxV_{m+\ell}\mxZ_c )$ is the orthogonal complement of the compressed solution space $\calR \left( \mxV_{m+\ell} \begin{bmatrix} \mxZ_1 & \mxZ_2 \end{bmatrix} \right)$
with respect to the (full) GKB solution space $\calR( \mxV_{m+\ell})$. An analogous relation holds for $\calR( \mxU_{m+\ell+1}\mxT_c )$.
Substituting these relations in (\ref{eq:new_gkb}) and using the fact that $\mxU_{m+\ell+1} 
  \left[ \mxT_1 \:\: \mxT_2 \right]$ has orthonormal columns, we obtain

\eqs
\nonumber
  \bfA \left[\bfW_k \:\: \wt{\bfV}_\ell  \right]  =
  \left[ \bfY_k \:\:  \wt{\bfU}_{\ell+1} \right]
  \begin{bmatrix} \bfR_k & \bfY_k\t\bfA \wt{\bfV}_\ell \\
    \bf0 & \wt{\bfB}_{\ell}
  \end{bmatrix}
  \eqv
  \bfA \mxV_{m+\ell} 
  \left[ \mxZ_1 \:\: \mxZ_2 \right] =
  \mxU_{m+\ell+1} 
  \left[ \mxT_1 \:\: \mxT_2 \right]
  \wh{\mxB}_{k,\ell} \quad   & \imp & \\
\label{eq:Bkl-Bml1}
  \wh{\mxB}_{k,\ell}  =
  \ars{cc} \mxT_1 & \mxT_2 \are\t \mxU_{m+\ell+1}\t \mxA
  \mxV_{m+\ell} \ars{cc} \mxZ_1 & \mxZ_2 \are =
   \ars{cc} \mxT_1 & \mxT_2 \are \t \mxB_{m+\ell} \ars{cc} \mxZ_1 & \mxZ_2 \are . &&
\eqe

Blockwise, we have $\mxT_1\t \mxB_{m+\ell} \mxZ_1 = \mxR_k$,
$\mxT_1\t \mxB_{m+\ell} \mxZ_2 = \mxY_k\t \mxA \wt{\mxV}_{\ell}$,
$\mxT_2\t \mxB_{m+\ell} \mxZ_1 = \mxO$, and
$\mxT_2\t \mxB_{m+\ell} \mxZ_2 = \wt{\mxB}_{\ell}$.
For the (3,1) block we have
\eqsn
  \mxT_c\t \mxB_{m+\ell} \mxZ_1 & = & \mxT_c\t\mxU_{m+\ell+1}\t  \mxA \mxV_{m+\ell} \mxZ_1 =
  \mxT_c\t \mxU_{m+\ell+1}\t \mxA \mxW_k =
  \mxT_c\t \mxU_{m+\ell+1}\t \mxY_k \mxR_k = \\
  &&
  \mxT_c\t \mxU_{m+\ell+1}\t \mxU_{m+\ell+1} \mxT_1 \mxR_k = \mxO .
\eqen
For the (3,2) block we have
\eqsn
  \mxT_c\t \mxB_{m+\ell} \mxZ_2 & = & \mxT_c\t\mxU_{m+\ell+1}\t  \mxA \mxV_{m+\ell} \mxZ_2 =
  \mxT_c\t \mxU_{m+\ell+1}\t \mxA \wt{\mxV}_{\ell} =
  \mxT_c\t \mxU_{m+\ell+1}\t (\mxY_k \mxY_k\t \mxA \wt{\mxV}_{\ell}
  + \wt{\mxU}_{\ell+1} \wt{\mxB}_{\ell}) = \\
  &&
  \mxT_c\t \mxU_{m+\ell+1}\t \mxU_{m+\ell+1}\ars{cc}\mxT_1 & \mxT_2 \are
  \begin{bmatrix} \mxY_k\t \mxA \wt{\mxV}_{\ell}  \\   \wt{\mxB}_{\ell}   \end{bmatrix}
  = \mxO .
\eqen
This gives the following Lemma.
\begin{lemma}\label{lem:TransBm+l}
\eqs \label{eq:TTBmlZ}
\ars{ccc} \mxT_1 & \mxT_2 & \mxT_c \are\t \mxB_{m+\ell} \ars{ccc} \mxZ_1 & \mxZ_2 & \mxZ_c \are
  & = &
  \begin{bmatrix}
  \bfR_k & \bfY_k\t\bfA \wt{\bfV}_\ell  & \mxT_1\t \mxB_{m+\ell} \mxZ_c \\
  \mxO      & \wt{\bfB}_{\ell}     &   \mxT_2\t \mxB_{m+\ell} \mxZ_c                 \\
  \mxO     &  \mxO                        &   \mxT_c\t \mxB_{m+\ell} \mxZ_c
  \end{bmatrix} .
\eqe
\end{lemma}

Next we consider the difference between the regularized solution to
(\ref{eq:projectedreg}) and the regularized solution to the
full (transformed) problem with system matrix (\ref{eq:TTBmlZ}).
In particular, we analyze the backward error, and then consider
bounds on the backward error for the special case
of compression based on the TSVD.
Let $\l$ be given, typically an appropriate $\l$ for the regularized problem
(\ref{eq:projectedreg}),  and let
$\wt{\bfy}_{\l}$ be given as in \eqref{eq:projectedreg}, i.e.,
\eqs \label{eq:RecyTichSystem}
\left( \wh{\mxB}_{k,\ell}\t\wh{\mxB}_{k,\ell} + \l^2\mxI\right)\wt{\bfy}_{\l}
  & = &
  \wh{\mxB}_{k,\ell}\t \begin{bmatrix}
                         \mxR_k\bfe_k + \bfzeta \\
                         \wt{\b}_1 \bfe_1 
                       \end{bmatrix} .
\eqe
We consider the residual of the approximate solution $\ars{cc} \wt{\bfy}\t_{\l} & \bfzero\t \are \t$
for the regularized (transformed) full problem
\eqsn
\left( \left(\mxT\t\mxB_{m+\ell}\mxZ\right)\t \left(\mxT\t\mxB_{m+\ell}\mxZ\right)
  + \l^2 \mxI\right) \bfy  = 
  \left(\mxT\t\mxB_{m+\ell}\mxZ\right)\t \mxT\t \bfe_1\b_1
  & = & \left(\mxT\t\mxB_{m+\ell}\mxZ\right)\t 
  \mxT\t \mxU_{m+\ell+1}\t \bfb  \\
= \left(\mxT\t\mxB_{m+\ell}\mxZ\right)\t
  \ars{ccc} \mxY_k & \wt{\mxU}_{\ell+1} & \mxU_{m+\ell+1}\mxT_c \are\t \bfb & = &
\left(\mxT\t\mxB_{m+\ell}\mxZ\right)\t
  \begin{bmatrix}
    \mxR_k \bfe_k + \bfzeta \\
    \wt{\b}_1 \bfe_1 \\
    \bfzero
  \end{bmatrix}
\eqen
The residual for $\ars{cc}\wt{\bfy}_{\l}\t & \bfzero\t \are \t$ for the full transformed problem is given by
\eqs \label{eq:ResComprProbl}
  \bfr_{\l} & = &
  \begin{bmatrix}
      \wh{\mxB}_{k,\ell}\t & \mxO \\
      \mxZ_c\t\mxB_{m+\ell}\t \ars{cc} \mxT_1 & \mxT_2 \are & \mxZ_c\t\mxB_{m+\ell}\t \mxT_c
  \end{bmatrix}
  \begin{bmatrix}
    \begin{bmatrix}
      \mxR_k \bfe_k + \bfzeta \\
      \wt{\b}_1 \bfe_1
    \end{bmatrix} \\
    \bfzero
  \end{bmatrix}
  -
\\ \nonumber
  &&
  \begin{bmatrix}
    \wh{\mxB}_{k,\ell}\t\wh{\mxB}_{k,\ell} + \l^2 \mxI & \wh{\mxB}_{k,\ell}\t \ars{cc}\mxT_1 & \mxT_2 \are \t \mxB_{m+\ell} \mxZ_c \\
    \mxZ_c\t \mxB_{m+\ell}\t \ars{cc}\mxT_1 & \mxT_2\are \wh{\mxB}_{k,\ell} &
    \mxZ_c\t \mxB_{m+\ell}\t\mxB_{m+\ell}\mxZ_c + \l^2 \mxI
  \end{bmatrix}
  \begin{bmatrix}
    \wt{\bfy}_{\l} \\ \bfzero
  \end{bmatrix}
\\ \nonumber
  &=&
  \begin{bmatrix}
    \bfzero \\
    \mxZ_c\t \mxB_{m+\ell}\t \ars{cc}\mxT_1 & \mxT_2 \are
    \left(
    \begin{bmatrix}
      \mxR_k \bfe_k + \bfzeta  \\
      \wt{\b}_1 \bfe_1
    \end{bmatrix} -
    \wh{\mxB}_{k,\ell} \wt{\bfy}_{\l}
    \right)
  \end{bmatrix} .
\eqe
Note that 
\begin{equation}
\label{eq:rhat}
\wh{\bfr}_{\l} = 
\begin{bmatrix} 
  \mxR_k \bfe_k + \bfzeta  \\ \wt{\b}_1 \bfe_1 
\end{bmatrix} -
\wh{\mxB}_{k,\ell} \wt{\bfy}_{\l}
\end{equation}
is just the residual for the regularized
solution of (\ref{eq:projectedreg}) with the chosen $\l$, and its
norm is known. The corresponding residuals for the full system are $[\bfY_k \; \wt{\bfU}_{\ell+1}] \, \wh{\bfr}_{\l}$, obtained with compression and recycling, and 
$[ \bfY_k \; \wt{\bfU}_{\ell+1} \; (\bfU_{m+\ell+1}\bfT_c)] \, \bfr_{\l}$, obtained with $m+\ell$ steps of standard GKB, but with the regularization parameter and solution from the
compression and recycling 
approach.
This gives the following theorem.
\begin{theorem} 
Let $\mxZ$, $\mxT$, $\mxB_{m+\ell}$, $\bfr_{\l}$, and $\wh{\bfr}_{\l}$ be defined as above. Then,
\eqsn
  \bfr_{\l} & = &
    \begin{bmatrix}
    \bfzero \\
    \mxZ_c\t \mxB_{m+\ell}\t \ars{cc}\mxT_1 & \mxT_2 \are
    \wh{\bfr}_{\l}
  \end{bmatrix}  ,
\\
  \| \bfr_{\l} \|_2 & = &
  \left\| \mxZ_c\t \mxB_{m+\ell}\t \ars{cc}\mxT_1 & \mxT_2 \are
    \wh{\bfr}_{\l} \right\|_2  .
\eqen
\end{theorem}

Before we analyze $\|{\bfr}_{\l}\|_2$ and what it means for the difference between
the solutions from the regularized compressed problem with recycling and the regularized full problem
for the same regularization parameter,
consider the case that this residual of the regularized full
problem is
(relatively) small.
In that case, the backward error is (relatively) small, and for a well-chosen regularization
parameter the matrix is well-conditioned. Hence the difference between the
regularized solution for the full problem and the regularized solution for the compressed
problem is small. Assuming the regularization parameter is larger than the smallest
singular values, the condition number of the regularized matrix depends
on the largest singular value and the regularization parameter. In general
(pathological cases excepted), $\s_{max}(\mxB_{m+l}) \approx \s_{max}(\wh{\mxB}_{k,\ell})$,
and hence choosing $\l$ such that the compressed Tikhonov problem is well-conditioned
implies that the full Tikhonov problem would be well-conditioned for the same
$\l$.

\subsection*{Analysis for TSVD-based compression}
Next, we consider a more detailed analysis in the case that compression is done using TSVD. For simplicity, we consider 
compression after the first $m$ iterations of standard GKB, so (\ref{eq:gk1}) is satisfied, and $\ell$ subsequent steps of recycling GKB. We can extend this to an analysis for multiple compression and recycling steps, but this is left for future work. 

Let $\bfB_m = {\bf \Psi}_{m+1}  {\bf \Sigma}_m {\bf \Phi}_m^\top$ be the SVD of $\bfB_m$ with 
\begin{equation}
    \label{eq:sig_Bm}
    {\bf\Sigma}_m = \text{diag}(\s_1,\dots,\s_m),
\end{equation}
and let $\bfx^{(1)}$ be a regularized solution.
We take $\bfW_{k} = \bfV_m \begin{bmatrix} {\bf \Phi}_{k-1} & {\bs \xi}\end{bmatrix}$, where, 
following section 
\ref{sec:recyclingGKB}, 
$\bfw_k = \widecheck{\bfx}^{(1)}$ and so 
${\bs \xi} = \bfV_m^T \widecheck{\bfx}^{(1)} $.
This gives
\eqsn
  \bfY_k\bfR_k = \bfA \bfW_k & = & \bfU_{m+1} \bfB_m \ars{cc} {\bs \Phi}_{k-1} & {\bs \xi} \are
  = \bfU_{m+1} \ars{cc} {\bs \Psi}_{k-1}{\bs \Sigma}_{k-1} &
         \wt{\bs \eta} \are ,
\eqen
with $\wt{\bs \eta} = \bfB_m {\bs \xi}$. Since ${\bs \xi} \perp {\bs \Phi}_{k-1}$, we have $\wt{\bs \eta} \perp {\bs \Psi}_{k-1}$, and for the QR decomposition
$\bfY_k \bfR_k = \bfA \bfW_k$,
\eqs \label{eq:TSVD-YR}
  \bfY_k =  
  \bfU_{m+1}
  \ars{cc}{\bs \Psi}_{k-1} & {\bs \eta} \are,  &\qquad&
  \bfR_k = 
  \diag{(\sigma_1, \ldots, \sigma_{k-1},r_{kk})} 
  \mbox{ with } 
  r_{kk} = \|\wt{\bs \eta}\|_2 \leq \sigma_k .
\eqe
Next, we bound $\|\bfr_{\l}\|_2$ by bounding
$\| \bfZ_c\t\bfB_{m+\ell}\t \ars{cc} \bfT_1 & \bfT_2 \are \|_F$,
which is an obvious upper bound for 
$\| \bfZ_c\t\bfB_{m+\ell}\t \ars{cc} \bfT_1 & \bfT_2\are \|_2$.
We note that this Frobenius norm bound is
computable a posteriori (without extra cost).
First, consider $\bfZ_c\t\bfB_{m+\ell}\t \bfT_1$.
Since $\calR(\bfA\t \bfY_k) \subset \calR(\bfV_{m+1})$
and $\bfZ_c\t\bfB_{m+\ell}\t \bfT_1 = 
\bfZ_c\t\bfV_{m+\ell}\t\bfA\t\bfY_k$, 
\eqsn
  \|\bfZ_c\t\bfB_{m+\ell}\t \bfT_1\|_F^2 & = &
  \|\bfA\t \bfY_k\|_F^2 - \|\bfZ_1\t\bfB_{m+\ell}\t \bfT_1\|_F^2
  - \|\bfZ_2\t\bfB_{m+\ell}\t \bfT_1\|_F^2 .
\eqen
We have
\eqsn
  \bfA\t\bfY_k & = & 
  \bfA\t \bfU_{m+1} \ars{cc} {\bs \Psi}_{k-1} & {\bs \eta} \are
  = \bfV_{m+1}\begin{bmatrix}
    \bfB_{m}\t \\ \bfe_{m+1}\t\a_{m+1}
  \end{bmatrix}
  \ars{cc}{\bs \Psi}_{k-1} & {\bs \eta}\are \;\; \imp
\\
  \|\bfA\t\bfY_k \|_F^2 & = &
  \|{\bs \Sigma}_{k-1}\|_F^2 + \|\bfB_{m}\t {\bs \eta}\|_2^2
  + \a_{m+1}^2\|\bfe_{m+1}\t \ars{cc} {\bs \Psi}_{k-1} & {\bs \eta} \are\|_2^2 \leq
  \s_1^2 + \cdots + \s_{k-1}^2 + \s_k^2 + \a_{m+1}^2 .
\eqen
Note that $\bfB_{m}\t {\bs \eta}$ and 
$\bfe_{m+1}\t \ars{cc} {\bs \Psi}_{k-1} & {\bs \eta}\are$ can be computed at 
negligible cost during the algorithm. Also,
$\bfZ_1\t\bfB_{m+\ell}\t \bfT_1 = \bfR_k\t$, which implies 
$\| \bfZ_1\t\bfB_{m+\ell}\t \bfT_1 \|_F^2 = 
\s_1^2 + \cdots + \s_{k-1}^2 + r_{kk}^2$ (with 
$r_{kk} = \| \bfB_m {\bs \xi} \|_2$), 
and
$\| \bfZ_2\t\bfB_{m+\ell}\t \bfT_1 \|_F^2 = 
\|\bfY_k\t\bfA \wt{\bfV}_{\ell} \|_F^2$, which also can be computed at negligible cost during the algorithm.
This gives 
\eqs\label{eq:normZcBT1}
  \|\bfZ_c\t\bfB_{m+\ell}\t \bfT_1\|_F^2 & = &
  \|\bfB_{m}\t{\bs \eta}\|_2^2 - \|\bfB_{m}{\bs \xi}\|_2^2
  + \a_{m+1}^2\|\bfe_{m+1}\t\ars{cc}{\bs \Psi}_{k-1} & {\bs \eta}\are\|_2^2 - \|\bfY_k\t\bfA \wt{\bfV}_{\ell} \|_F^2
\\
  & \leq & \s_k^2 - r_{kk}^2 + \a_{m+1}^2 - \|\bfY_k\t\bfA \wt{\bfV}_{\ell} \|_F^2 .
\eqe
Note that $\bfeta = \mxB_m \bfxi/ \|\mxB_m \bfxi\|_2$, and hence
$\|\bfB_{m}\t{\bs \eta}\|_2^2 - \|\bfB_{m}{\bs \xi}\|_2^2$ tends to be small.
For $\bfZ_c\t\bfB_{m+\ell}\t \bfT_2$, we have
\eqsn
  \bfZ_c\t\bfB_{m+\ell}\t \bfT_2 & = &
  \bfZ_c\t\bfV_{m+\ell}\t\bfA\t\bfU_{m+\ell+1} \bfT_2 = 
  \bfZ_c\t\bfV_{m+\ell}\t\bfA\t\wt{\bfU}_{\ell+1}
\\
  & = & 
  \bfZ_c\t\bfV_{m+\ell}\t \left(\wt{\bfV}_{\ell} \wt{\bfB}_{\ell}\t
  + \wt{\a}_{\ell+1}\wt{\bfv}_{\ell+1}\bfe_{\ell+1}\t\right)
  = \wt{\a}_{\ell+1}\bfZ_c\t\bfV_{m+\ell}\t\wt{\bfv}_{\ell+1}\bfe_{\ell+1}\t , 
\eqen
and hence
\eqs\label{eq:normZcBT2}
  \|\bfZ_c\t\bfB_{m+\ell}\t \bfT_2\|_F & = &
  |\wt{\a}_{\ell+1}| \|\bfZ_c\t\bfV_{m+\ell}\t\wt{\bfv}_{\ell+1}\|_2
  \leq |\wt{\a}_{\ell+1}| .
\eqe
This derivation proves the following theorem.
\begin{theorem}
\label{thm:bound}
Let $\bfA$, $\bfb$, $\bfV_{m+\ell}$,
$\bfU_{m+\ell+1}$, $\bfB_{m+\ell}$, and 
$\wh{\bfB}_{k,\ell}$ be defined as above (using 
TSVD for compression).
Then $\bfr_{\l}$ in (\ref{eq:ResComprProbl}) satisfies 
\begin{equation}
    \label{eq:bound}
\| \bfr_{\l} \|_2  \leq 
  \|\wh{\bfr}_{\l}\|_2
  \left(\s_k^2 - r_{kk}^2 + \a_{m+1}^2 - \|\bfY_k\t\bfA \wt{\bfV}_{\ell} \|_F^2 + \wt{\a}_{\ell+1}^2\right)^{1/2},
\end{equation}
where $\wh{\bfr}_{\l}$ is given in \eqref{eq:rhat}.
\end{theorem}

Finally, we provide a few more useful properties of the matrix $\bfT\t \bfB_{m+\ell} \bfZ$. 
First, we would like a bound on its maximum singular value, so we can estimate the condition number of the regularized
matrix in (\ref{eq:ResComprProbl}). Note that its
smallest eigenvalue is larger than $\l^2$, where 
$\l$ is the regularization parameter. Using $\s_1(\cdot)$ to denote the largest singular value of a matrix, 
\begin{align}
  \s_1\left(\bfT\t \bfB_{m+\ell} \bfZ\right)
  & \leq 
  \s_1\left(\bfT\t \bfB_{m+\ell} \ars{cc}\bfZ_1 & \bfZ_2\are
    \right) + \s_1(\bfT\t \bfB_{m+\ell} \bfZ_c) \\
  & \leq 
  \s_1\left(\bfT\t \bfB_{m+\ell} \ars{cc} \bfZ_1 & \bfZ_2 \are\right) + \|\bfT\t \bfB_{m+\ell} \bfZ_c\|_F .
\end{align}
Estimating $\bfT_c\t\bfB_{m+\ell}\bfZ_c$
is difficult, and we use a conjecture that we test numerically below.
Consider the blocks of $\bfB_{m+\ell}$,
\eqs \label{eq:Bm+l-blocks}
\bfB_{m+\ell} = \begin{bmatrix}\bfB_{m} & \begin{bmatrix}\bfzero_{m\times \ell} \\ \alpha_{m+1}\bfe_1\t \end{bmatrix} \\\\ \bfzero_{\ell \times m} &  \overline{\bfB}_{\ell}\end{bmatrix} 
& \;\;\mbox{ with }\;\; &
\overline{\mxB}_{\ell} =\begin{bmatrix}
		\beta_{m+2} & \alpha_{m+2}& & \\
		 & \ddots& \ddots & \\
		 & & \beta_{m+\ell} & \alpha_{m+\ell} \\
		 & & & \beta_{m+\ell+1}  \\
	\end{bmatrix}\in \bbR^{\ell \times \ell} .
\eqe
We also define 
\eqs \label{eq:Blbar2}
\overline{\overline{\bfB}}_{\ell} = \begin{bmatrix}\alpha_{m+1}\bfe_1\t \\\\ \overline{\bfB}_{\ell}\end{bmatrix}.
\eqe
Since
$\|\bfB_m\|_F^2 + 
\|\overline{\overline{\bfB}}_{\ell}\|_F^2 = 
\|\bfB_{m+\ell}\|_F^2 = 
\|\bfT\t\bfB_{m+\ell}\bfZ\|_F^2 = 
\|\bfT\t\bfB_{m+\ell}\ars{cc}\bfZ_1 & \bfZ_2 \are\|_F^2 +
\|\bfT\t\bfB_{m+\ell}\bfZ_c\|_F^2$,
we get
\eqs\label{eq:TTBmlZc_est}
\|\bfT\t\bfB_{m+\ell}\bfZ_c\|_F^2 & = &
\|\bfB_m\|_F^2 - \|\bfR_k\|_F^2 - \|\bfY_k\t\bfA\wt{\bfV}_{\ell}\|_F^2 +
\left(\|\overline{\overline{\bfB}}_{\ell}\|_F^2 - 
\|\wt{\bfB}_{\ell}\|_F^2\right) .
\eqe
Since the sequences of vectors $\wt{\bfv}_1, \wt{\bfv}_2, \ldots$ and $\bfv_{m+1}, \bfv_{m+2}, \ldots$ both
extend the Krylov space beyond $\calR(\bfV_m)$, and 
both sequences are orthogonal to the approximate
dominant 
singular vectors $\bfW_k$, we conjecture,
\begin{conj} \label{conj1}
  $\|\overline{\overline{\bfB}}_{\ell}\|_F^2 \approx 
\|\wt{\bfB}_{\ell}\|_F^2 .$
\end{conj}

\noindent
We provide numerical confirmation of this conjecture 
at the end of this section.
Using Conjecture \ref{conj1}, (\ref{eq:TTBmlZc_est}), (\ref{eq:sig_Bm})
and (\ref{eq:TSVD-YR}), we obtain the following (approximate) bound, 
\eqs 
\nonumber
\|\bfT\t\bfB_{m+\ell}\bfZ_c\|_F^2 & = &
  \s_k^2 - r_{kk}^2 + \s_{k+1}^2 + \ldots \s_m^2  - \|\bfY_k\t\bfA\wt{\bfV}_{\ell}\|_F^2
  + \left(\|\overline{\overline{\bfB}}_{\ell}\|_F^2 - 
\|\wt{\bfB}_{\ell}\|_F^2\right) 
\\
\label{eq:appr-bnd-TBZc}
& \lessapprox &
\s_k^2 - r_{kk}^2 + \s_{k+1}^2 + \ldots + \s_m^2 .
\eqe
This gives the following approximate bound,
\eqs
  \s_1(\bfB_{m+\ell}) = \s_1(\bfT\t \bfB_{m+\ell} \bfZ)
  & \lessapprox & 
  \left( \s_1^2\left(\begin{bmatrix}
    \bfR_k & \bfY_k\t\bfA\wt{\bfV}_{\ell} \\
    \bfzero & \wt{\bfB}_{\ell} \\
    \bfzero & \bfzero
  \end{bmatrix}\right) + 
  \s_k^2 - r_{kk}^2 + \s_{k+1}^2 + \ldots + \s_m^2
  \right)^{1/2} ,
\eqe
which in general is only modestly larger than $\s_1$.
Note that the first term on the right hand side, 
which equals $\s_1^2(\wh{\bfB}_{k,\ell})$, 
is easily computable. 

Using (\ref{eq:appr-bnd-TBZc}), 
(\ref{eq:normZcBT1}) - (\ref{eq:normZcBT2}), 
and $\s_k \geq \|\bfB^T \bs \eta\|_2 \geq r_{kk}$, we also obtain
an approximate bound for the bottom right block
of $\bfT\t \bfB_{m+\ell} \bfZ$, 
\begin{align}
\nonumber
  \|\bfT_c\t\bfB_{m+\ell}\bfZ_c\|_F^2 &  = 
  \s_k^2 -r_{kk}^2 + \s_{k+1}^2 + \ldots + \s_m^2 -
  \|\bfY_k^T\bfA\wt{\bfV}_{\ell}\|_F^2 
  + \left(\|\overline{\overline{\bfB}}_{\ell}\|_F^2 -
  \|\wt{\bfB}_{\ell}\|_F^2\right) 
\\
\nonumber
  & \qquad
  - \|\bfB_m\t{\bs \eta}\|_2^2 + r_{kk}^2 -
  \a_{m+1}^2\|\bfe_m\t\ars{cc}{\bs \Psi}_{k-1}& {\bs \eta}\are\|_2^2
  + \|\bfY_k^T\bfA\wt{\bfV}_{\ell}\|_F^2
  - \wt{\a}_{\ell+1}^2\|\bfZ_c\t\bfV_{m+\ell}\t\wt{\bfv}_{\ell+1}\|_2^2
\\
\label{eq:TcBZv-bnd}
  & \lessapprox
  \s_k^2 -r_{kk}^2 + \s_{k+1}^2 + \ldots + \s_m^2 .
\end{align}

Finally, we note that, in general, as the dominant
singular vectors are captured relatively quickly and progressively
better in the Krylov spaces, the coefficients
$|\a_j|$ and $|\b_{j+1}|$ have a decreasing trend.
Hence, $|\a_{m+1}|$ and $|\wt{\a}_{\ell+1}|$ tend to be small compared with $\s_k$.
In Figure \ref{pic:upperbound}, we numerically verify this trend for $\alpha_j$ for
the example in \cref{sec:id}.

\begin{figure}[h!]
\centering
\includegraphics[width = 0.85\textwidth]{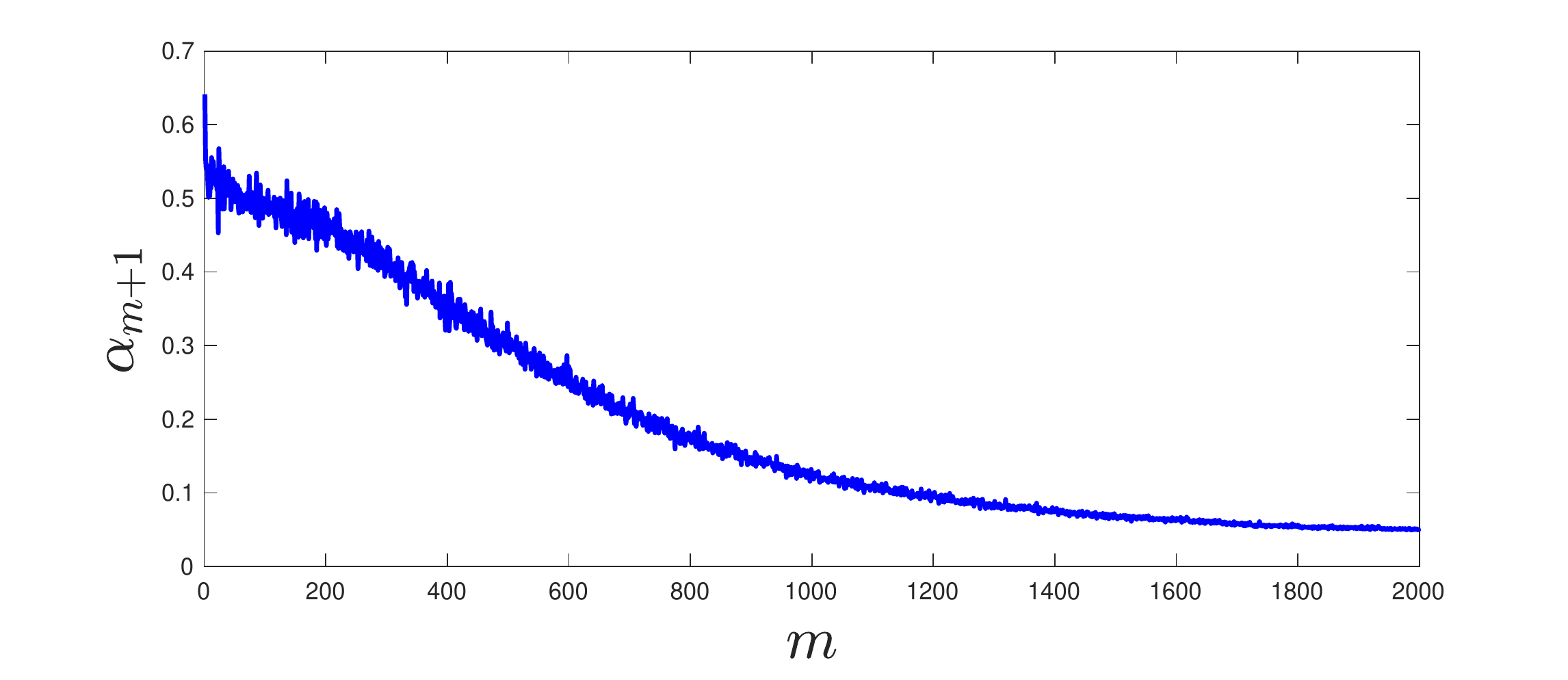}
\caption{\em Values of $\alpha_{m+1}$ for
the example in \cref{sec:id}.}
\label{pic:upperbound}
\end{figure}

To get a better understanding of the final term in the bound, we provide the difference
\begin{equation*}
d = \left|\norm[F]{\overline{\overline{\bfB}}_{\ell}}^2 - \norm[F]{\wt{\bfB}_{\ell}}^2\right|
\end{equation*}
for some commonly used numerical examples from RestoreTools \cite{nagy2004iterative}. In \cref{table1} we consider different images; \cref{table2} we consider different types of blur; and in \cref{table3} we consider different noise levels. We observe that $d$ is much smaller than $\s_k$ in regular scenarios, so $\s_k$ dominates the upper bound of $ \|\bfT_c\t\bfB_{m+\ell}\bfZ_c\|_F^2$ in  (\ref{eq:TcBZv-bnd}). 
\begin{table}[h]
   \centering 
   \begin{tabular}{lcccc} 
   \hline
   \hline
   \textbf{Image} & $d$ & $\norm[F]{\overline{\overline{\bfB}}_{\ell}}^2$ & $\norm[F]{\wt{\bfB}_{\ell}}^2$ & $\sigma_{30}$\\
   \hline
   Grain & 0.0593 & 9.3606 & 9.3013 & 0.6280\\
   \hline
   Plane &0.1345 & 9.4065 & 9.2720 & 0.6332\\
   \hline
   Peppers & 0.1548 & 9.4280 & 9.2732 & 0.6286\\
   \hline
   Cameraman& 0.1261 &9.4353 & 9.3092 & 0.6283\\
   \hline
   \end{tabular}
   \caption{Algorithm setting: $m = 50, k = q = 30, \ell=20$. Noise level: $0.2\%$. Size: $256 \times 256$. Blurring information: Gaussian with a missing piece.}
   \label{table1}
\end{table}

\begin{table}[h]
   \centering 
   \begin{tabular}{lcccc} 
   \hline
   \hline
   \textbf{Blurring information} & $d$ & $\norm[F]{\overline{\overline{\bfB}}_{\ell}}^2$ & $\norm[F]{\wt{\bfB}_{\ell}}^2$ & $\sigma_{30}$ \\
   \hline
    Gaussian with a missing piece & 0.0593 &9.3606 & 9.3013 & 0.6280\\
   \hline
   Nonsymmetric Gaussian blur with parameters $[2,3,0]$ &0.1094 &9.4610 & 9.3515&0.6153 \\
   \hline
    Nonsymmetric Out of Focus blur with radius $9$ & 0.1134 &9.4804 &9.3670 &0.6160\\
   \hline
   Box car blur ($11\times 11$ blur) & 0.1440 &9.4977 &9.3538 & 0.6155 \\
   \hline
   \end{tabular}
   \caption{Algorithm setting: $m = 50, k = q= 30, \ell=20$. Noise level: $0.2\%$. Size: $256 \times 256$. Image: Grain.}
   \label{table2}
\end{table}

\begin{table}[h]
   \centering 
   \begin{tabular}{lcccc} 
   \hline
   \hline
   \textbf{Noise level} & $d$ & $\norm[F]{\overline{\overline{\bfB}}_{\ell}}^2$ & $\norm[F]{\wt{\bfB}_{\ell}}^2$ & $\sigma_{30}$\\
   \hline
   0.005 & 0.0586 & 9.3599 & 9.3013 & 0.6280\\
   \hline
   0.01 & 0.0580 & 9.3593 & 9.3013 & 0.6279\\
   \hline
   0.05 & 0.0592 & 9.3605 & 9.3013 & 0.6275\\
   \hline
   0.1& 0.0605 & 9.3618 & 9.3013 & 0.6271\\
   \hline
   \end{tabular}
   \caption{Algorithm setting: $m = 50, k = q = 30, \ell=20$. Image: grain. Size: $256 \times 256$. Blurring information: Gaussian with a missing piece.}
   \label{table3}
\end{table}

\section{Numerical results}
\label{sec:numerics}
In this section, we compare the performance of the proposed hybrid projection methods with recycling to that of the conventional hybrid methods using examples from image processing. We consider various scenarios where the recycling hybrid projection methods can alleviate storage requirements and improve reconstructions when solving inverse problems.  In \cref{sec:id} we consider a linear image deblurring problem where standard hybrid methods may be limited by the storage of many vectors in the solution space.  We investigate the performance of various compression methods and parameter selection methods.  Then in \cref{sec:tomo}, we consider two tomographic reconstruction examples, one for a streaming data problem and another for a problem with modified projection angles using real data.

\subsection{Image deblurring example}
\label{sec:id}
This example is an image deblurring problem from RestoreTools \cite{nagy2004iterative}, where the goal is to reconstruct a true image of a grain, which has $256 \times 256$ pixels, from an observed blurred image that contains Gaussian white noise at a noise level of $0.2 \% $, i.e.,  $\frac{\norm[2]{\bfepsilon}}{\norm[2]{\bfA\bfx_{\rm true}}} = 0.002$.  The true image, blurred and noisy image, and point spread function (PSF) for the grain example are shown in Figure \ref{pic:grain}. An image deblurring problem with a smaller noise level usually requires more iterations to converge, which for standard hybrid methods means that we need to store more solution vectors. For this example, assume that we can store at most $50$ solution basis vectors, each of size $65536 \times 1$. We will show that the proposed hybrid projection method with recycling and compression, henceforth denoted \texttt{HyBR-recycle}, can handle this scenario. We will investigate various compression techniques from \cref{sec:compression_approaches}, where the stopping criteria are (1) the maximum number of basis vectors saved after compression is $q=30$ and (2) the compression tolerance is $\varepsilon_{tol} = 10^{-6}$. 
\begin{figure}[h!]
\begin{tabular}{ccc}
\includegraphics[width = 0.33\textwidth]{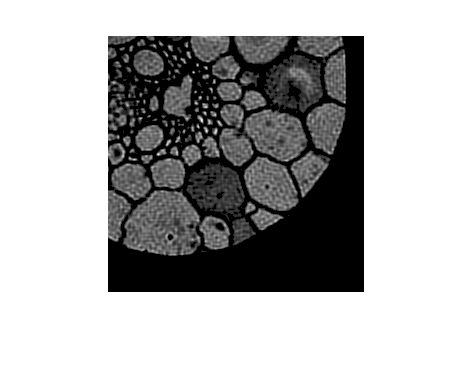}&
\includegraphics[width = 0.33\textwidth]{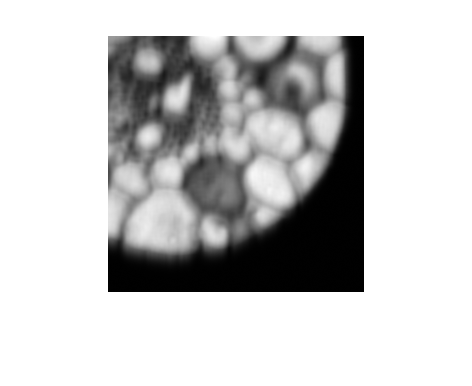}&
\includegraphics[width = 0.33\textwidth]{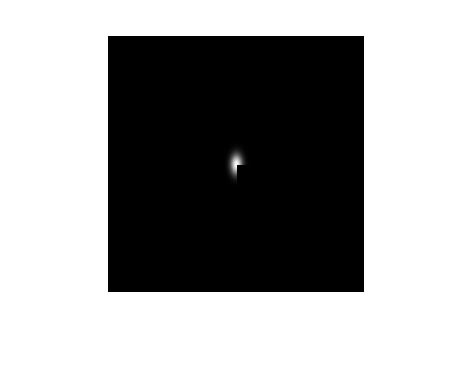}\\
(a) True image & (b) Noisy blurred image & (c) PSF
\end{tabular}
\caption{Image deblurring example.}
\label{pic:grain}
\end{figure}

In Figure \ref{pic:opt-wgcv}, we provide the relative reconstruction error norms for HyBR-recycle with various compression strategies, where for comparison we include the relative error norms for LSQR (with no additional regularization) and for a standard hybrid method denoted by \texttt{HyBR}.
In the left plot, we use the optimal regularization parameter at each iteration, which is not available in practice, and in the right plot, we use the weighted GCV (WGCV) method for regularization parameter selection. 
We observe that LSQR exhibits semiconvergence, and HyBR-opt is not able to achieve high accuracy due to the fact that the storage limit has been set to $50$ solution vectors. The reconstructions for hybrid projection methods with recycling and compression demonstrate the competitiveness of this approach in limited storage situations. 
Moreover,  we notice that for both regularization parameter choice methods, solution-oriented compression and sparsity-enforcing compression provide slightly smaller relative error norms than TSVD and RBD.

\begin{figure}[h]
\begin{tabular}{cc}
\includegraphics[width = 0.5\textwidth]{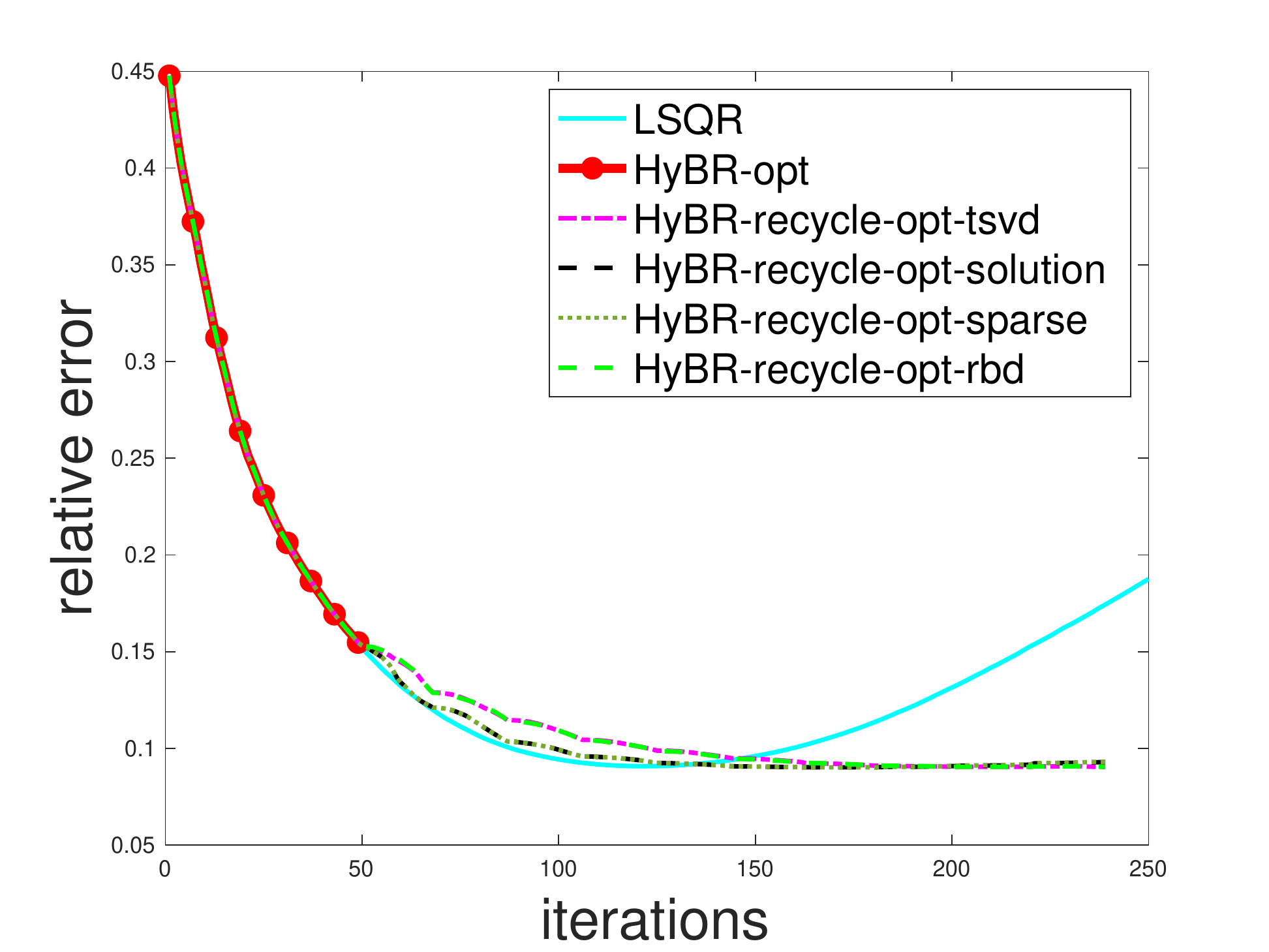}&
\includegraphics[width = 0.5\textwidth]{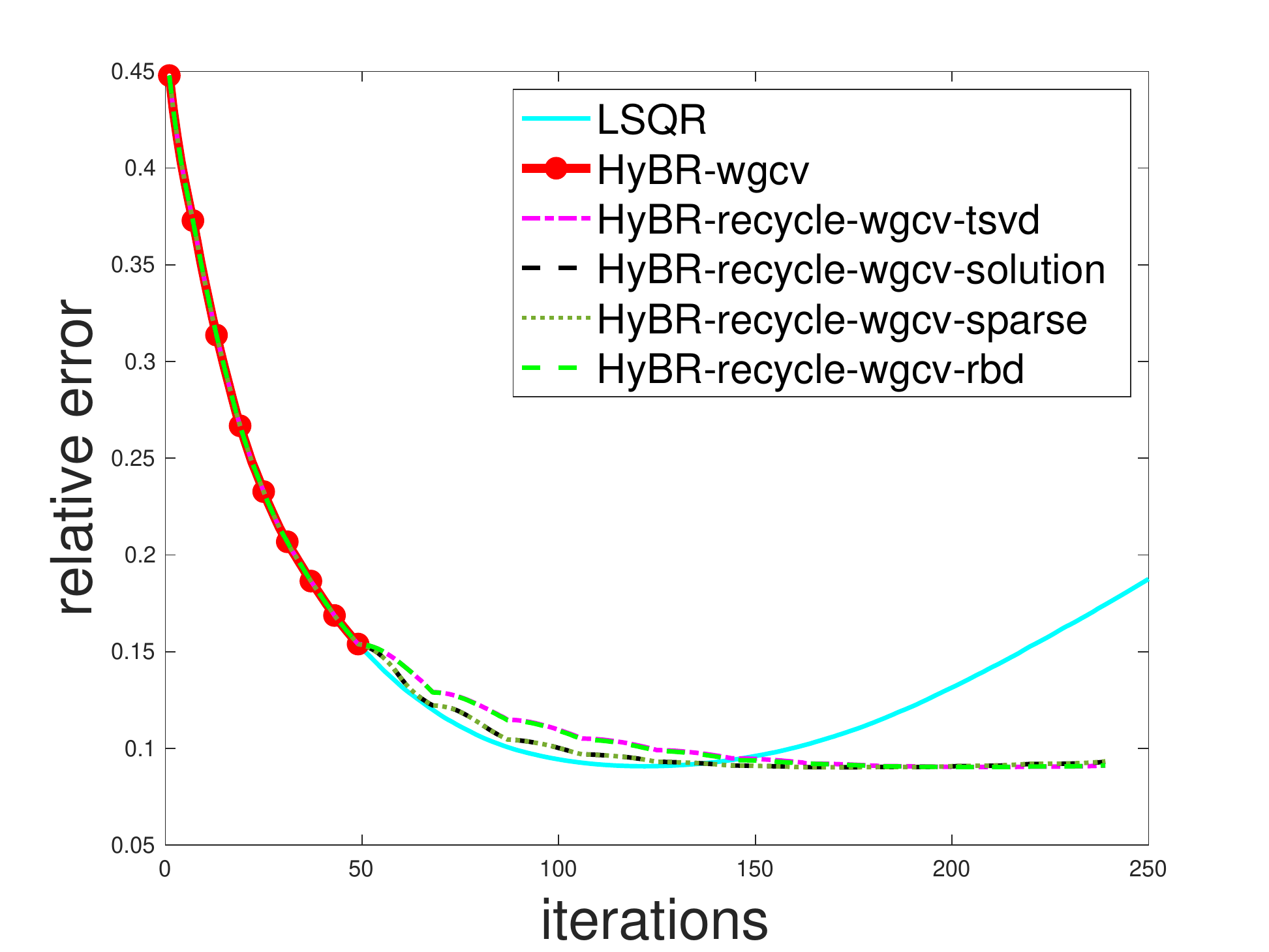}\\
(a) Optimal & (b) WGCV
\end{tabular}
\caption{Relative reconstruction error norms for hybrid projection methods with recycling using different compression strategies. The left plot corresponds to selecting the optimal regularization parameter at each iteration, and the right plot corresponds to selecting the regularization parameter using WGCV. For all methods, we assume that storage of solution basis vectors is limited to $50$.}
\label{pic:opt-wgcv}
\end{figure}

Next, in Figure \ref{pic:svd-wgcv} we compare various methods for selecting regularization parameters in hybrid projection methods with recycling. We consider two compression techniques (TSVD and solution-oriented), and we provide relative reconstruction error norms for parameter choice methods: the WGCV, the unbiased predictive risk estimator (UPRE), and the discrepancy principle (DP).
 For the experiments, we use the true noise level for both UPRE and DP, but estimates of the noise level can be obtained in practice.
For comparison, we provide results for the optimal regularization parameter.  We observe that all of the considered regularization parameter selection methods result in relative reconstruction error norms that are close to those for the optimal regularization parameter.

\begin{figure}[h!]
\begin{tabular}{cc}
\includegraphics[width = 0.5\textwidth]{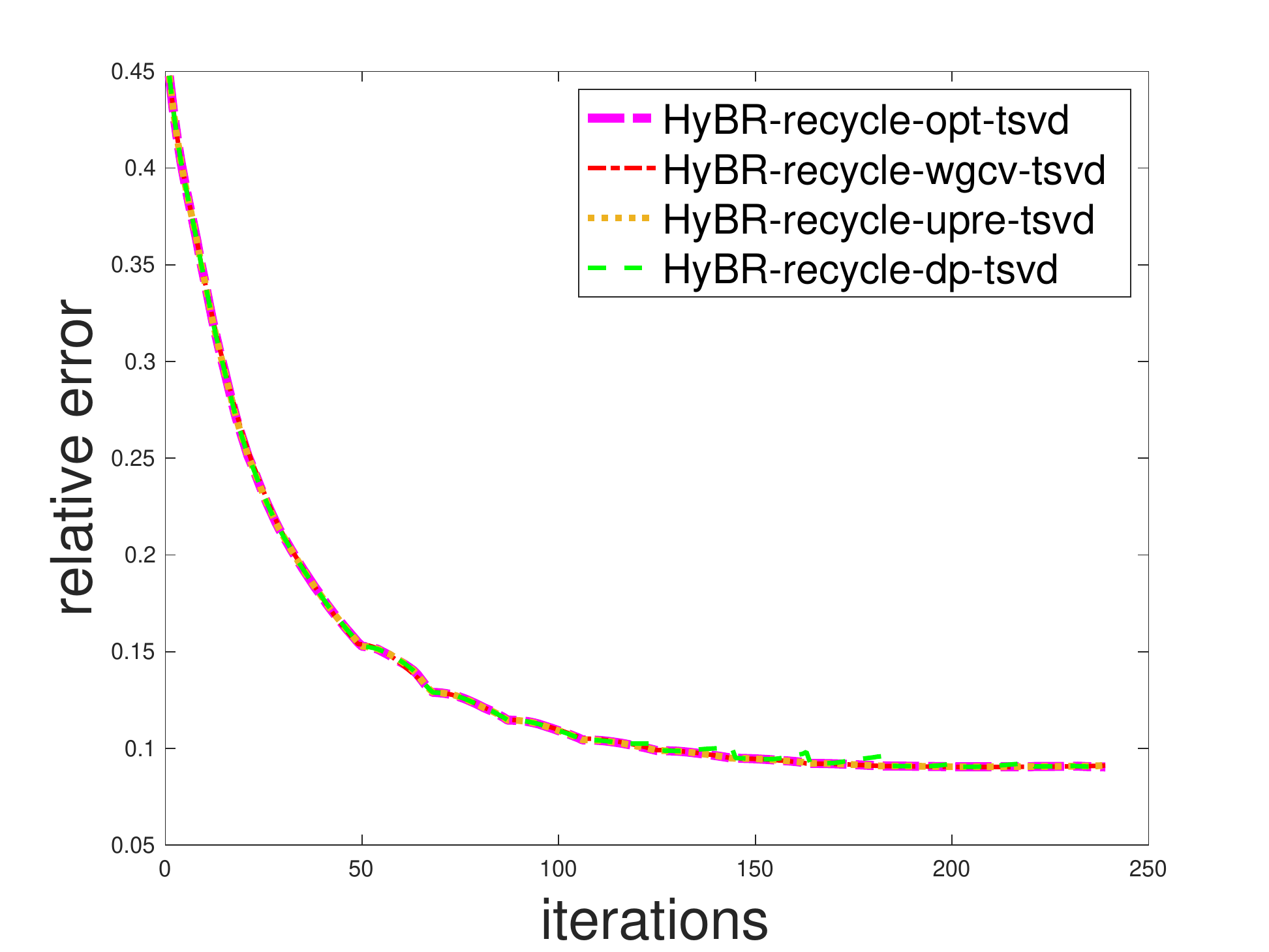}&
\includegraphics[width = 0.5\textwidth]{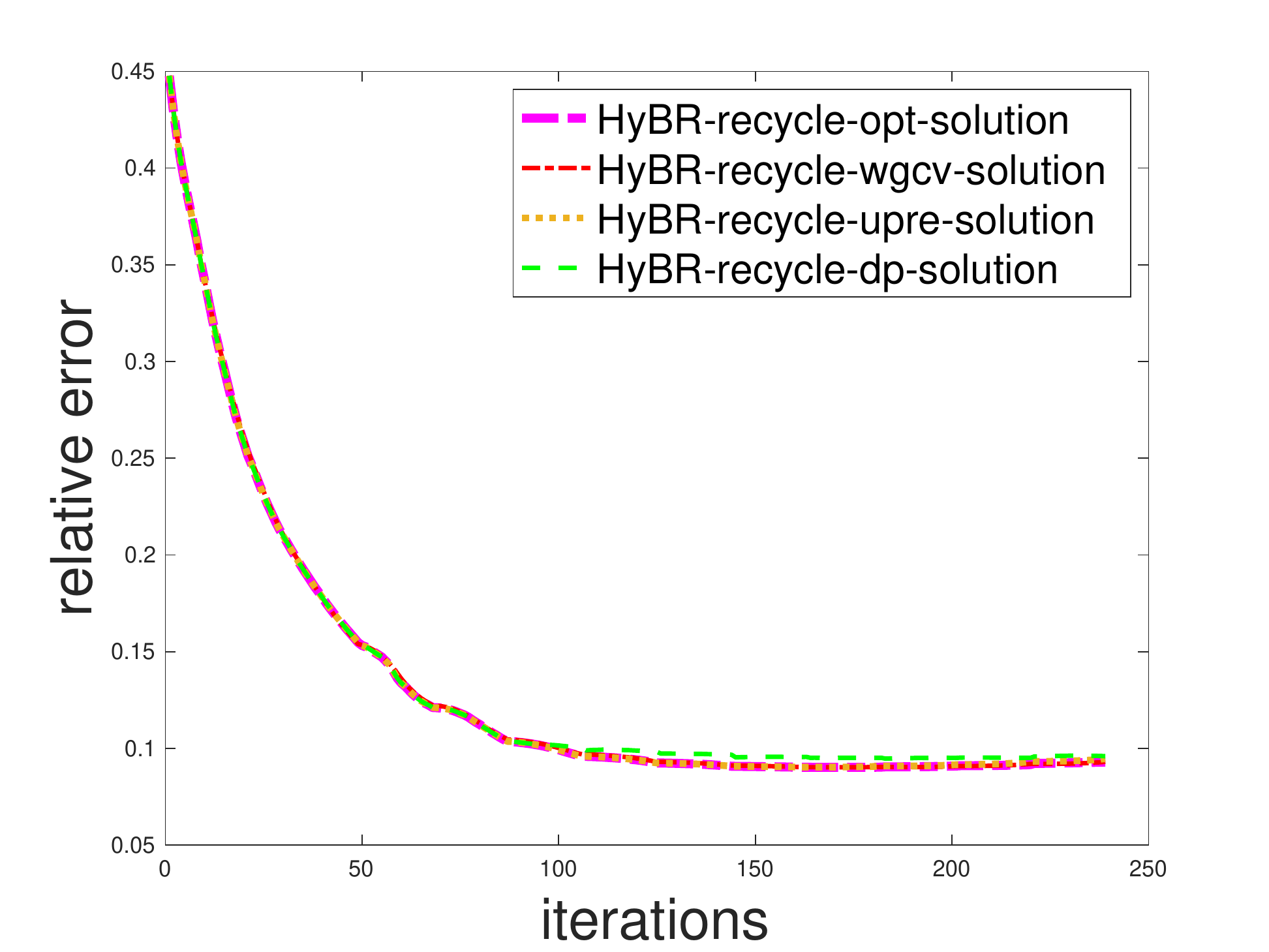}\\
(a) TSVD & (b) Solution-oriented method
\end{tabular}
\caption{Relative reconstruction errors for hybrid projection methods with recycling and compression for the grain deblurring example with different regularization parameter choice methods.}
\label{pic:svd-wgcv}
\end{figure}

The absolute error images (in inverted colormap) corresponding to reconstructions of the grain image are provided in Figure \ref{pic:deblur_image}. We compare reconstructions with standard hybrid methods after 50 iterations with reconstructions with hybrid projection methods with recycling after 239 iterations, for both WGCV and DP. For HyBR-recycle-WGCV, we provide results for TSVD and solution-oriented compression. For HyBR-recycle-DP, we provide results for RBD and sparsity enforcing compression. Due to the forced storage limit, HyBR-WGCV and HyBR-DP reconstruction absolute errors are large (corresponding to darker regions in Figures \ref{pic:deblur_image} (a) and (d) respectively). 
These observations are consistent with the relative error norms provided in Figure \ref{pic:opt-wgcv}.

\begin{figure}[h!]
\begin{tabular}{ccc}
\includegraphics[width = 0.33\textwidth]{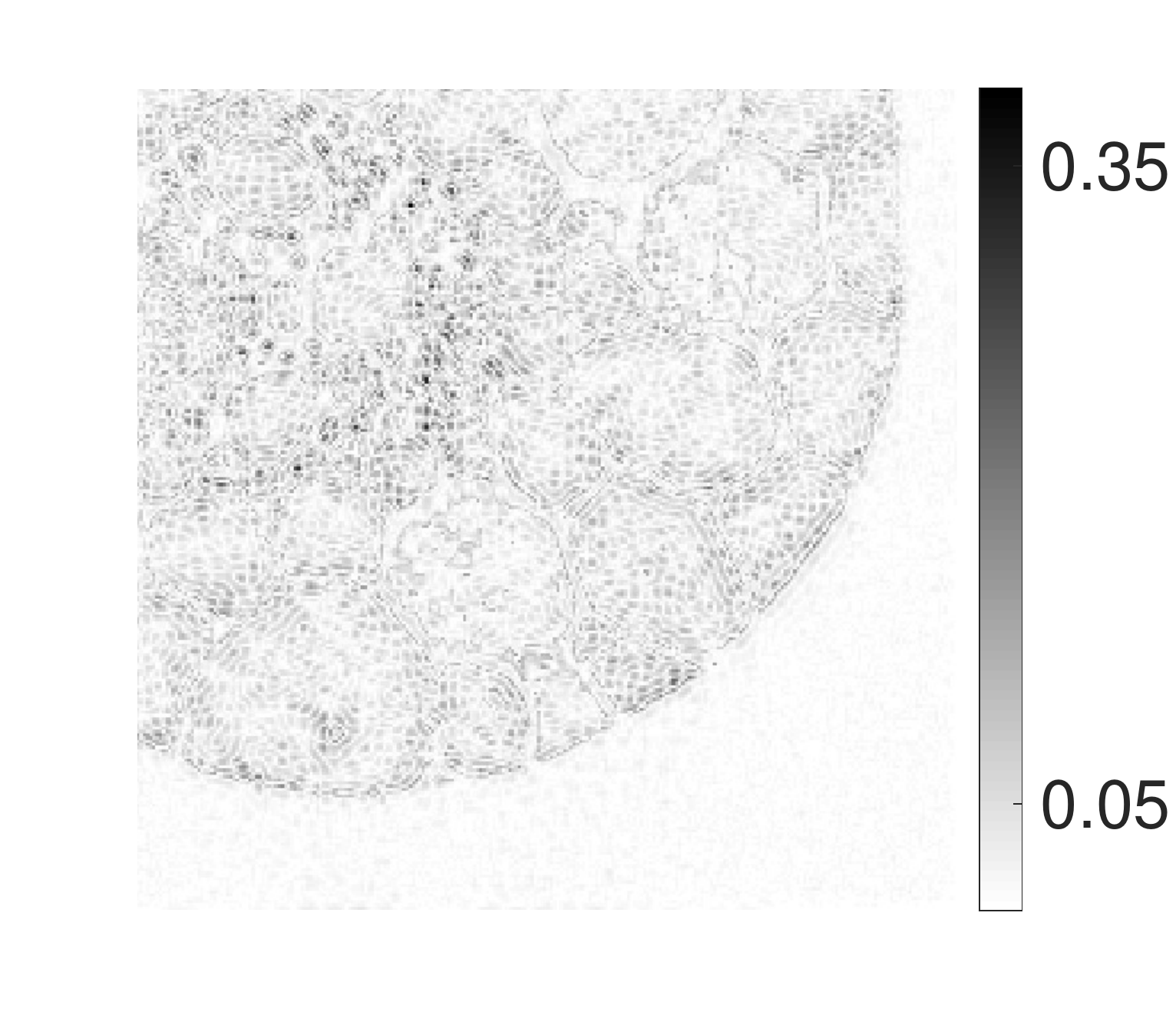}&
\includegraphics[width = 0.33\textwidth]{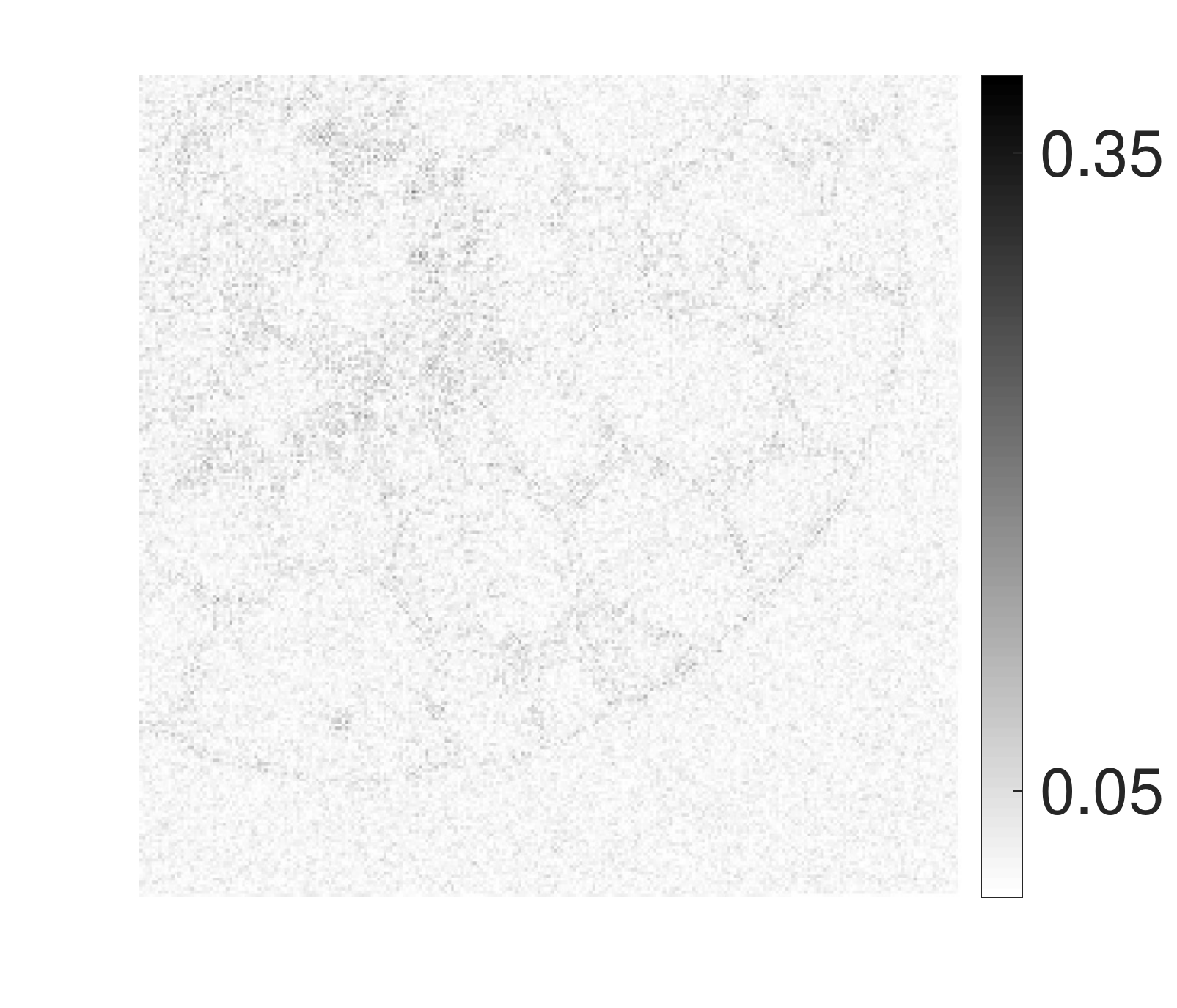}&
\includegraphics[width = 0.33\textwidth]{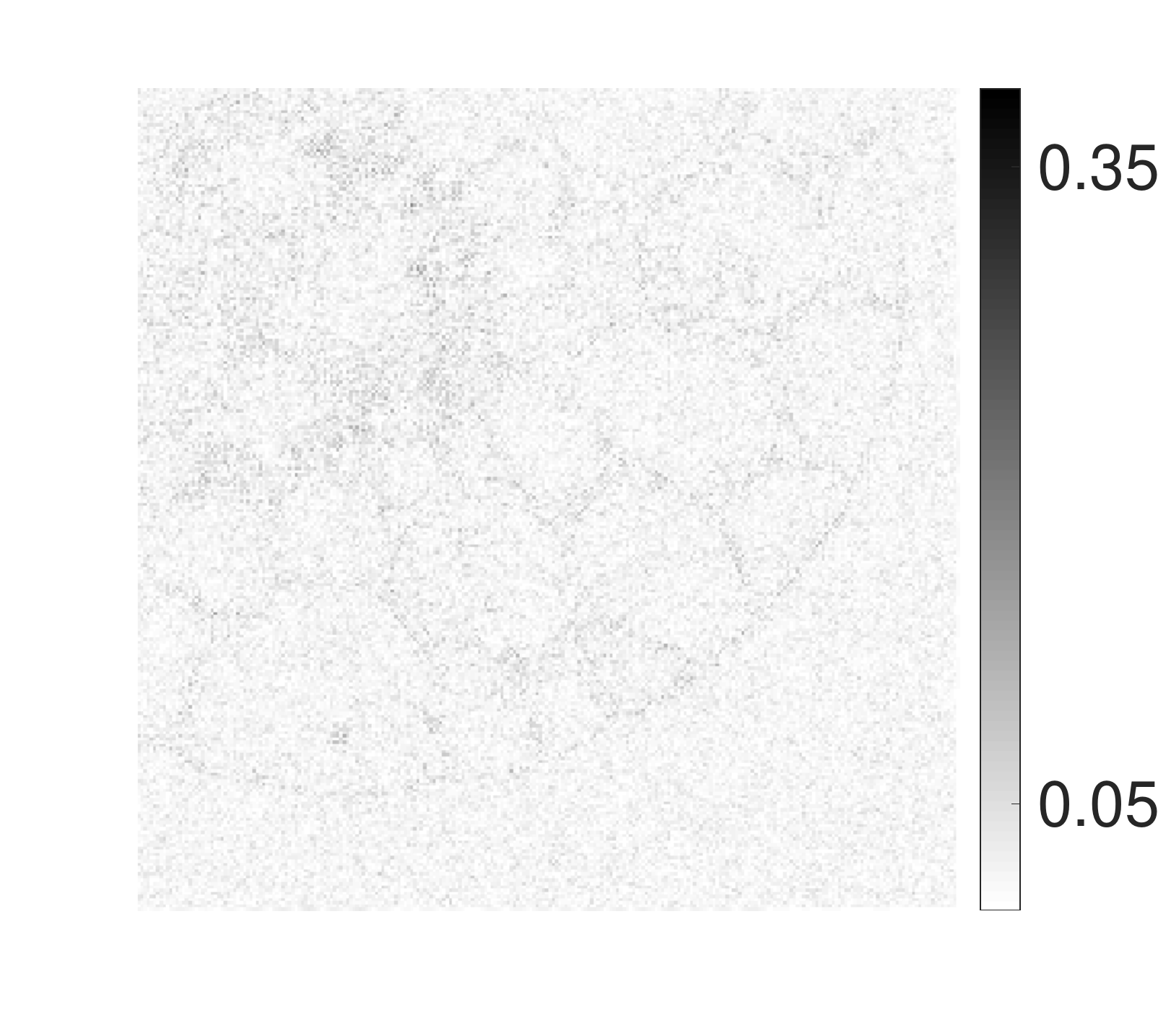}\\
(a)HyBR-WGCV  & (b) HyBR-recycle-WGCV-TSVD  & (c) HyBR-recycle-WGCV-solution \\
\includegraphics[width = 0.33\textwidth]{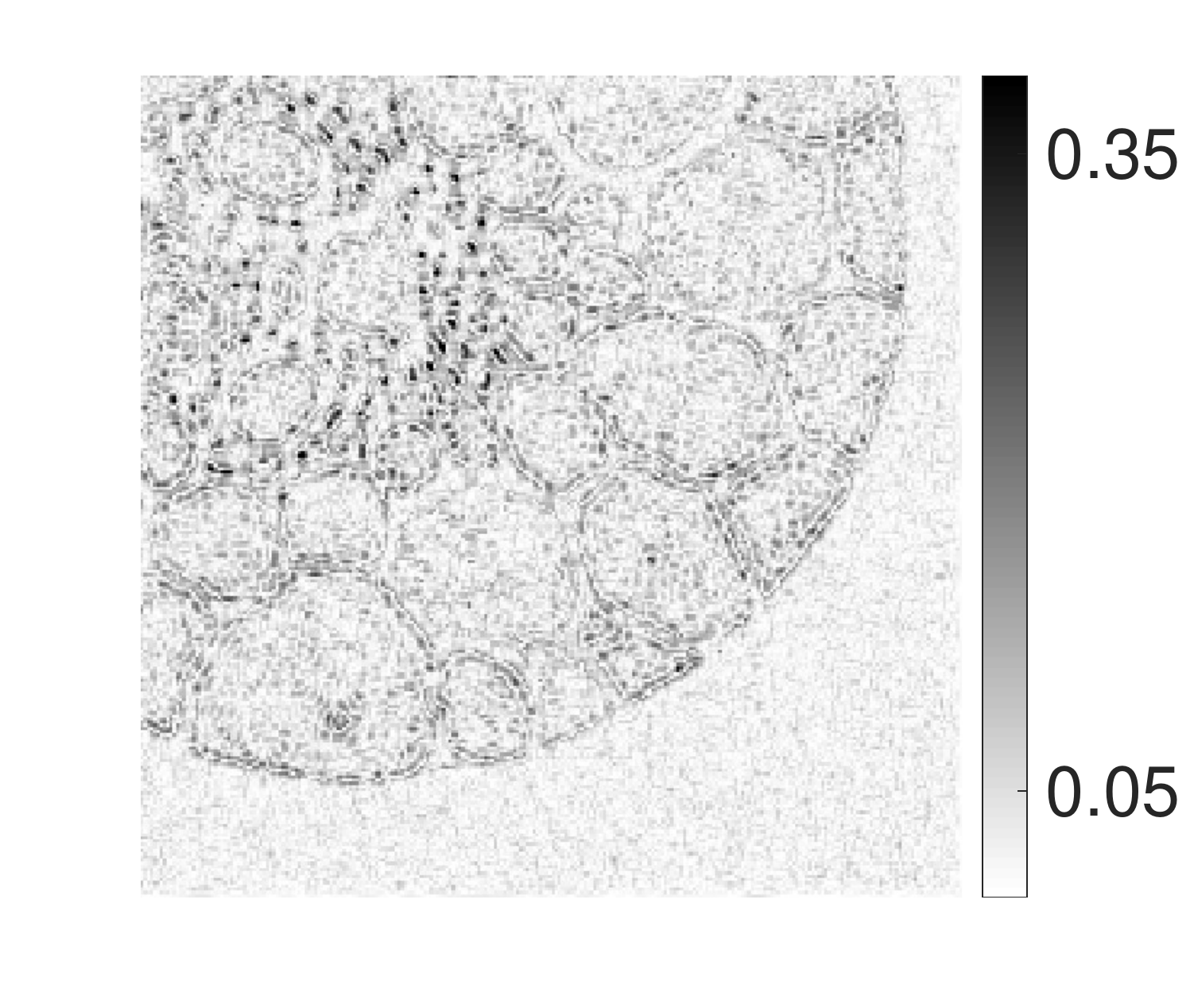}&
\includegraphics[width = 0.33\textwidth]{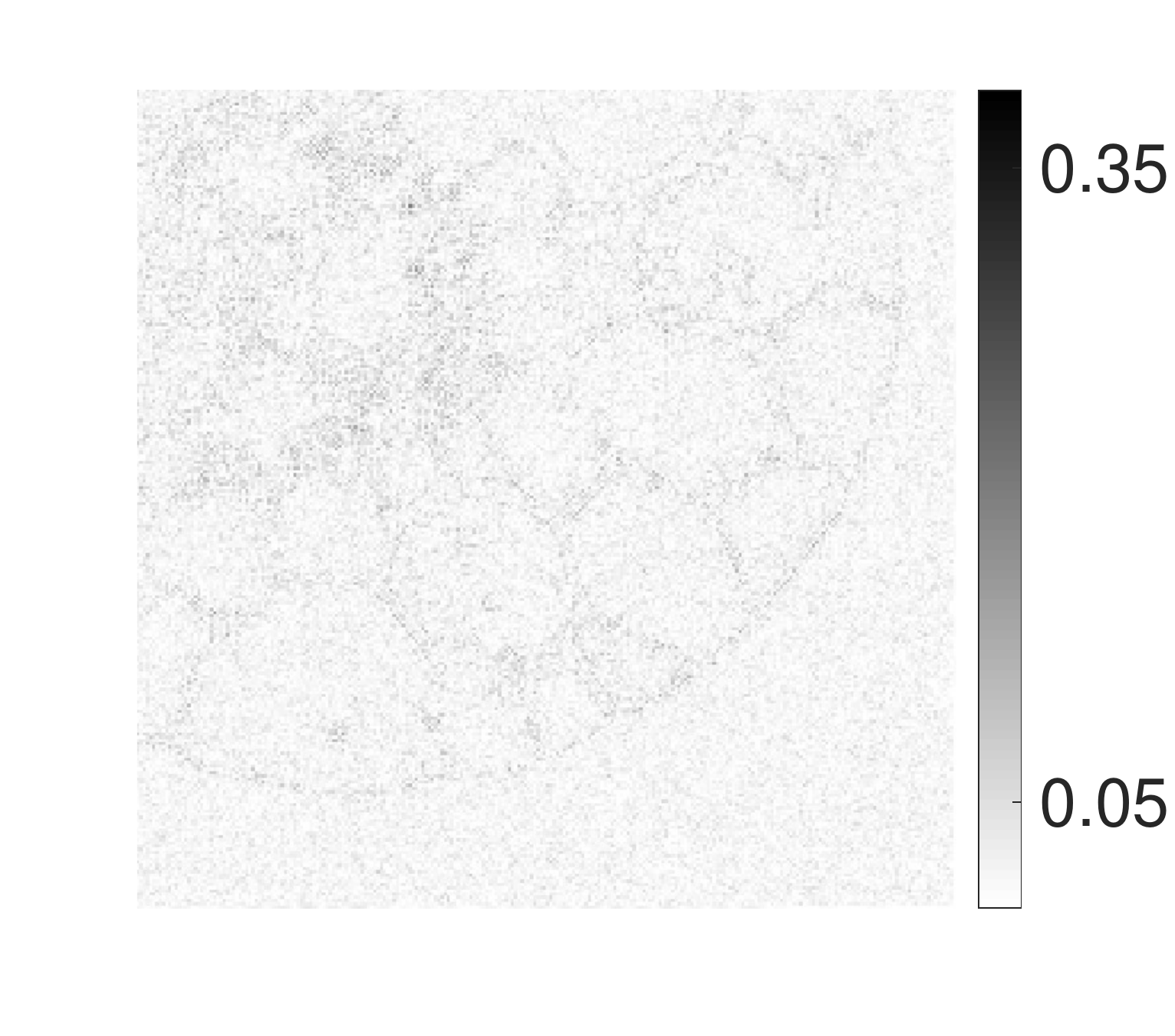}&
\includegraphics[width = 0.33\textwidth]{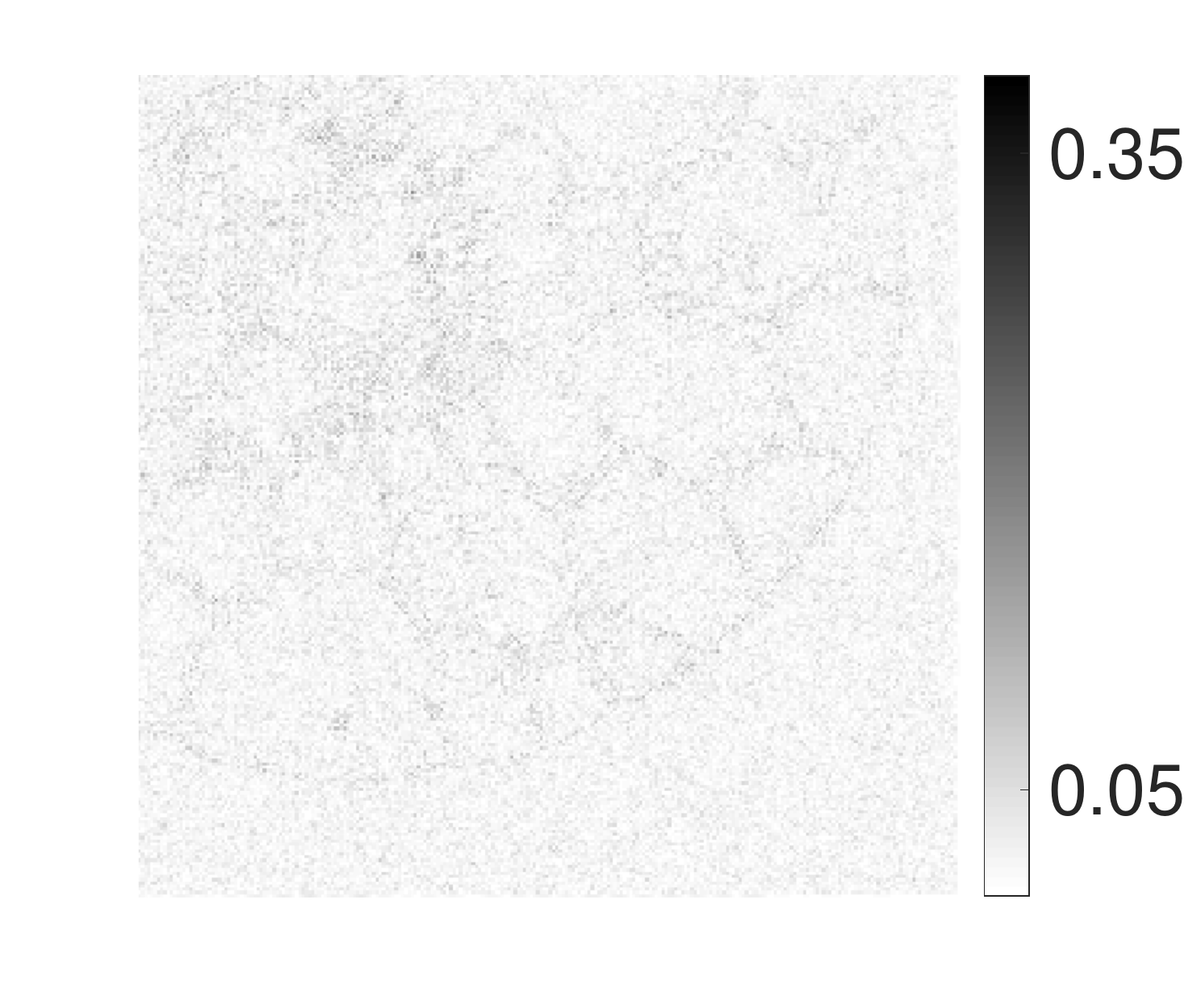}\\
(d)HyBR-DP  & (e) HyBR-recycle-DP-RBD  & (f) HyBR-recycle-DP-sparse \\
\end{tabular}
\caption{Absolute error images (in inverted colormap) for the grain image for WGCV and DP.}
\label{pic:deblur_image}
\end{figure}

Finally, we use this example to numerically demonstrate the bound derived in Theorem \ref{thm:bound}. In Figure \ref{fig:bound}, we provide the norm of the residual for the transformed problem and the derived upper bound from \eqref{eq:bound} for one cycle of HyBR-recycle after compression with TSVD.  At each iteration, the same regularization parameter was used to compute the residual from HyBR-recycle (i.e., $\norm[2]{\widehat{\bfr}_\lambda}$) and the residual from the regularized full GKB for the HyBR-recycle solution (i.e., $\norm[2]{\bfr_\lambda}$). Although the bound is an overestimate, the result shows that we do not expect the solution of HyBR-recycle after compression to be far from the solution to the regularized Tikhonov problem when using the full GKB.
\begin{figure}[h!]
    \centering
     \includegraphics[width = \textwidth]{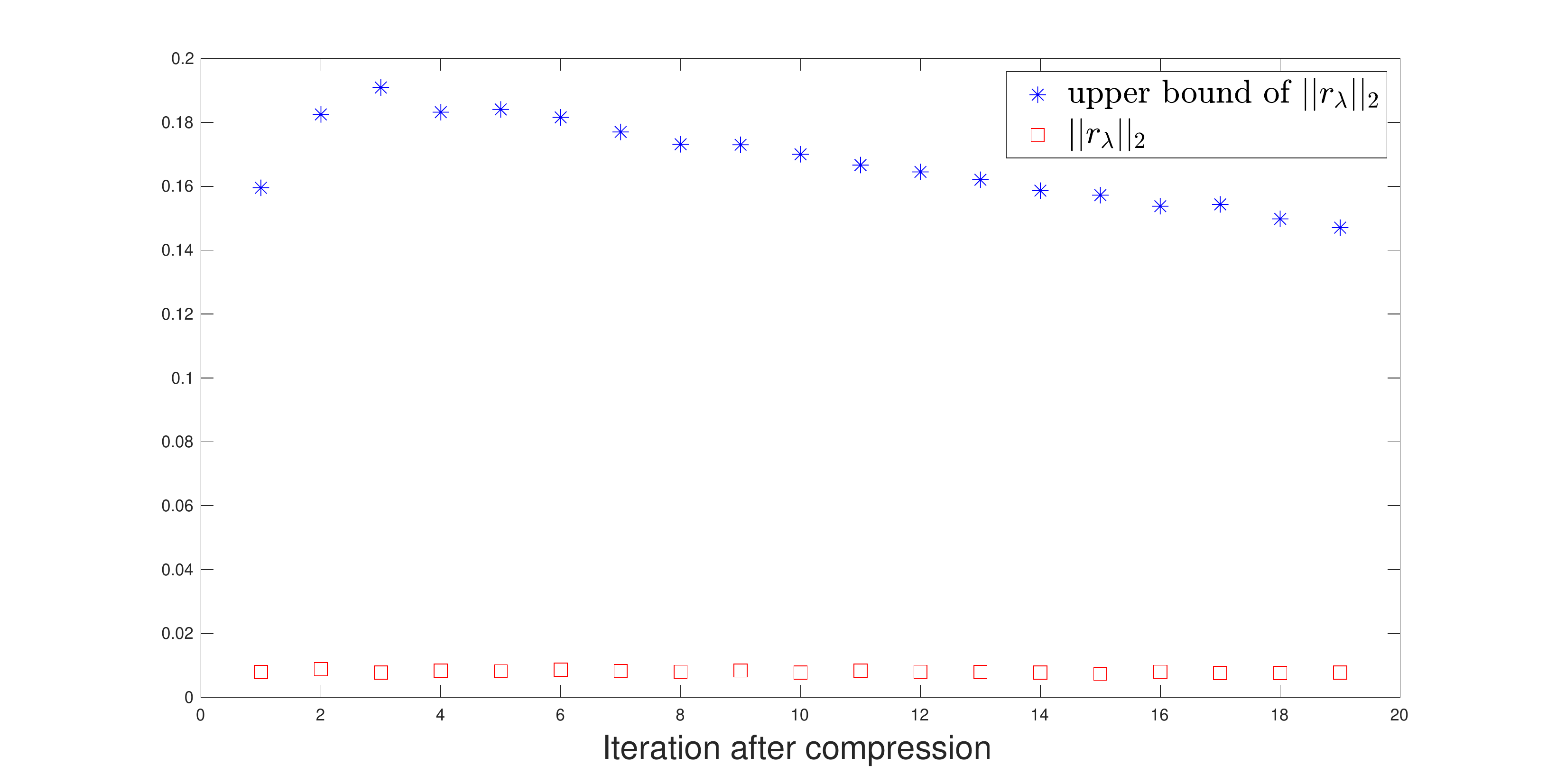}
    \caption{Illustration of bound on the residual norm derived in Theorem \ref{thm:bound}.}
    \label{fig:bound}
\end{figure}

\subsection{Tomography reconstruction examples}
\label{sec:tomo}
Next, we investigate various scenarios in tomographic reconstruction where multiple reconstruction problems must be solved, and the hybrid projection methods with recycling can be used to incorporate information (e.g., basis vectors) from previous reconstructions to solve the current reconstruction problem.
We consider two scenarios.
\begin{enumerate}
  \item In the case of dynamic or streaming data inverse problems, reconstructions must be updated as data are being collected.  This may arise in applications such as microCT, where immediate reconstructions are used as feedback to inform the data acquisition process \cite{parkinson2017machine}.
  \item Oftentimes, we must solve several reconstruction problems where the projection angles are slightly modified.  This might arise in an optimal experimental design framework where the goal is to determine the optimal angles for image formation \cite{ruthotto2018optimal} or in a sampling framework \cite{slagel2018sampled}.
\end{enumerate}

Before describing the details of the experiments, we describe four general approaches.
Assume that we have $r$ reconstruction problems,
\begin{align}
  \label{eq:tomo1}
    \min_\bfx \norm[2]{\bfA_1 \bfx - \bfb_1}^2 &+ \lambda_1^2 \norm[2]{\bfx}^2. \\
&\vdots \nonumber \\
   \label{eq:tomoi}
\min_\bfx \norm[2]{\bfA_i \bfx - \bfb_i}^2 &+ \lambda_i^2 \norm[2]{\bfx}^2. \\
\vdots \nonumber \\
 \label{eq:tomor}
\min_\bfx \norm[2]{\bfA_r \bfx - \bfb_r}^2 &+ \lambda_r^2 \norm[2]{\bfx}^2.
 \end{align}
 Depending on the problem setup and noise level, the regularization parameter for each problem $\lambda_1, \dots, \lambda_r$ may be different.  Thus, in all of our approaches, we select these regularization parameters automatically in a hybrid framework. 
\begin{enumerate}
  \item Using $\bfA_1$ and $\bfb_1,$ we run $m$ iterations of the standard Golub-Kahan bidiagonalization on \cref{eq:tomo1}, compress the computed solution vectors into $k_1-1$ orthonormal vectors $\bfW_{k_1-1} \in \mathbb{R}^{N \times (k_1-1)}$, and save matrix $\bfW_{k_1} = \begin{bmatrix}\bfW_{k_1-1} & \widecheck{\bfx}^{(1)}\end{bmatrix}$, where $$\widecheck{\bfx}^{(1)} = (\bfx^{(1)} - \bfW_{k_1 -1}\bfW_{k_1-1}^{\top}\bfx^{(1)})/\norm[2]{\bfx^{(1)} - \bfW_{k_1 -1}\bfW_{k_1-1}^{\top}\bfx^{(1)}}$$ and $\bfx^{(1)}$ is the corresponding solution of \cref{eq:tomo1}. Then for a subsequent problem with $\bfA_i$ and $\bfb_i$ ($1 < i < r$), we use $\bfW^{(new)}_{k_i - 1}$ obtained from the previous problem and run HyBR-recycle on \cref{eq:tomoi} saving matrix $\bfW^{(new)}_{k_i}$. Finally we solve \cref{eq:tomor} 
 using HyBR-recycle starting with $\bfW^{(new)}_{k_r - 1}$.
 \item We run a standard HyBR method with automatic regularization parameter selection on any of the $r$ reconstruction problems (e.g., the last one).
  \item For comparison, we provide the results for HyBR with automatic regularization parameter selection on the entire problem,
  \begin{equation}
    \label{eq:tomoproblemall}
    \min_\bfx \norm[2]{\begin{bmatrix}\bfA_1\\ \vdots \\ \bfA_r \end{bmatrix} \bfx - \begin{bmatrix} \bfb_1 \\ \vdots \\ \bfb_r \end{bmatrix}}^2 + \lambda^2 \norm[2]{\bfx}^2.
  \end{equation}
  We remark that in streaming scenarios, this can be considered as the ideal case and should produce the solution with overall smallest relative error.  However, we assume that this cannot be computed in practice and use it merely as a comparison.
  \item We take an average of solutions computed from \cref{eq:tomo1} to \cref{eq:tomor} independently; this is common in tomography.
\end{enumerate}

\subsubsection{Streaming data}
\label{sec:shepp2}
For the first experiment, we use the parallel tomography example from IRTools \cite{gazzola2017ir,hansen2018air}, where the true image is a $1024 \times 1024$ Shepp-Logan phantom so $\bfx_\true \in \bbR^{1024^2}$. The true image can be found in~\cref{fig:tomo_ex1} (a). We test two cases for this example, where the first case has two reconstruction problems ($r=2$) and the second one has four reconstruction problems ($r=4$).

{\bf Case 1:} We assume that data is being streamed such that the first reconstruction problem corresponds to $90$ equally spaced projection angles between $0^{\circ}$ and $89^{\circ}$, and the second problem corresponds to $90$ equally spaced projection angles between $90^{\circ}$ and $179^{\circ}$.  In terms of dimensions, $\bfA_1,\bfA_2 \in \bbR^{90\cdot1448 \times 1024^2}$ and $\bfb_1, \bfb_2 \in \bbR^{90\cdot1448}$. The noise level for each observed image is $0.02,$ which means that $\frac{\norm[2]{\bfepsilon_i}}{\norm[2]{\bfA_i \bfx_\true}} = 0.02$ for $i=1,2$.
The observations are provided in \cref{fig:tomo_ex1}. The limit of the storage of solution basis vectors (each $1,048,576 \times 1$) is assumed to be $50$. The stopping criteria for compression are defined by the maximum number of the basis vectors we want to keep after each compression, which we assume here to be $10$, and a tolerance, which we assume to be $\varepsilon_{tol} = 10^{-6}$.
\begin{figure}[htbp]
\centering
  \begin{tabular}{ccc}
  \includegraphics[width = .2\textwidth]{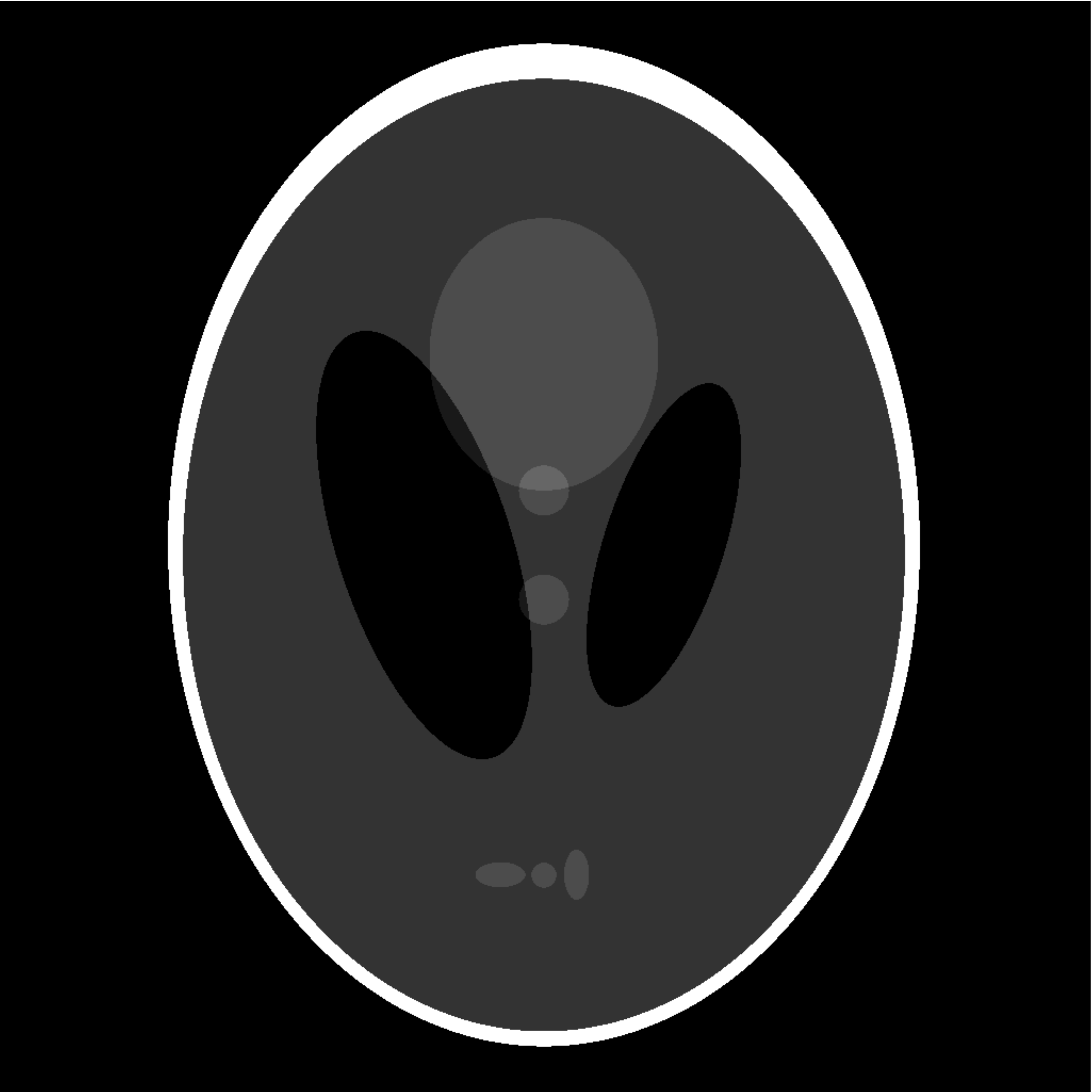} &
  \includegraphics[width = .4\textwidth]{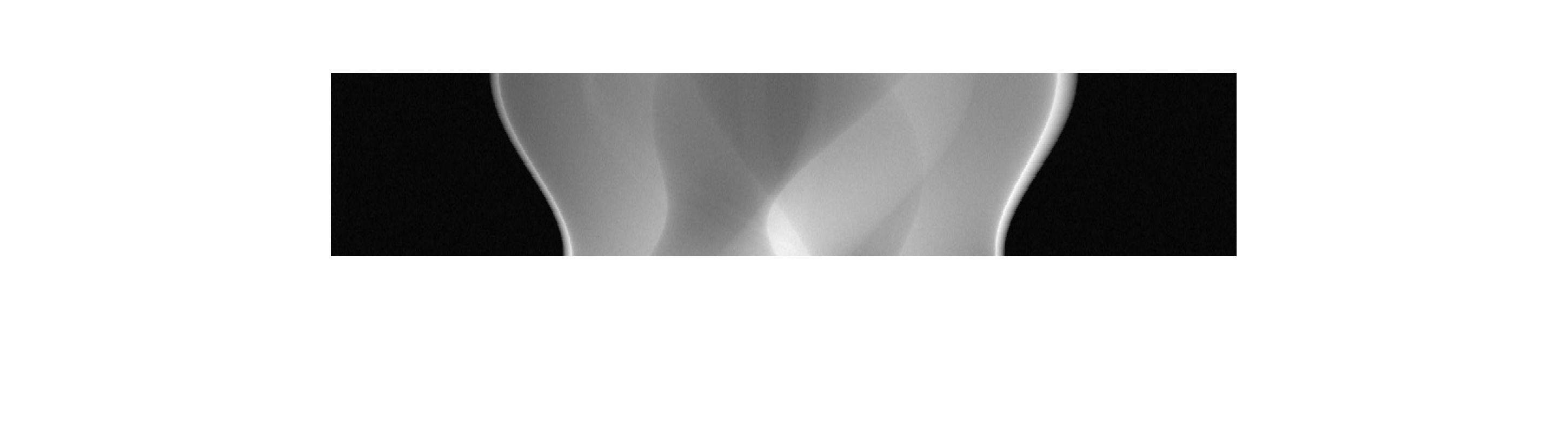} &
  \includegraphics[width = .4\textwidth]{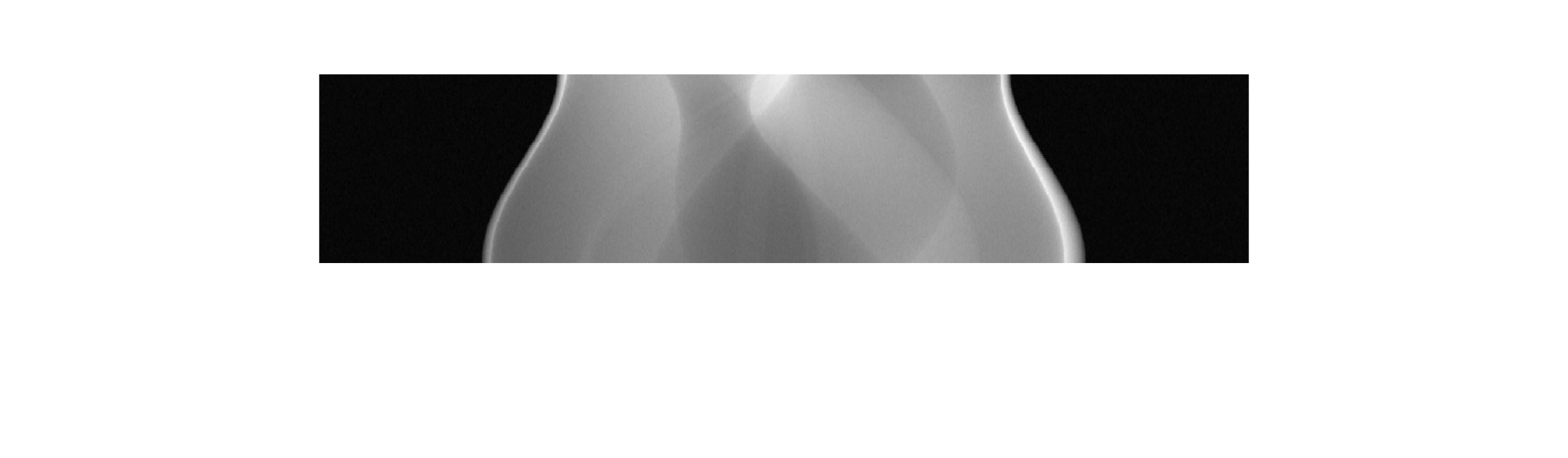} \\
(a) True image & (b) $\bfb_1$: $0^{\circ}-89^{\circ}$ & (c) $\bfb_2$: $90^{\circ}-179^{\circ}$
\end{tabular}
  \caption{Streaming tomography example, case 1. The true image is provided in (a), along with two observed sinograms $\bfb_1, \bfb_2$ corresponding to projections taken at $1^\circ$ intervals from $0^{\circ}-89^{\circ}$ and $90^{\circ}-179^{\circ}$ respectively.}
  \label{fig:tomo_ex1}
\end{figure}

A plot of the relative reconstruction error norms per iteration for the four approaches described above is provided in \cref{fig:tomorel2}. The average of images is not an iterative process, but the relative error norm corresponding to the average solution is denoted with a dotted line, for comparison.  We see that HyBR-recycle produces reconstructions with relative reconstruction error norms that are smaller than both HyBR with the 2nd dataset and the average of images, demonstrating that the inclusion of the $11$ basis images in the HyBR-recycle framework was beneficial.  Notice that HyBR with all of the data produces reconstructions with smallest reconstruction error norm, as expected. We also compare to the HEB approach without regularization.
The main point of this comparison is to demonstrate that HEB is not as accurate as HyBR-recycle since the generated basis is not improved.

\begin{figure}[htbp]
\centering
  \includegraphics[width = 0.8\textwidth]{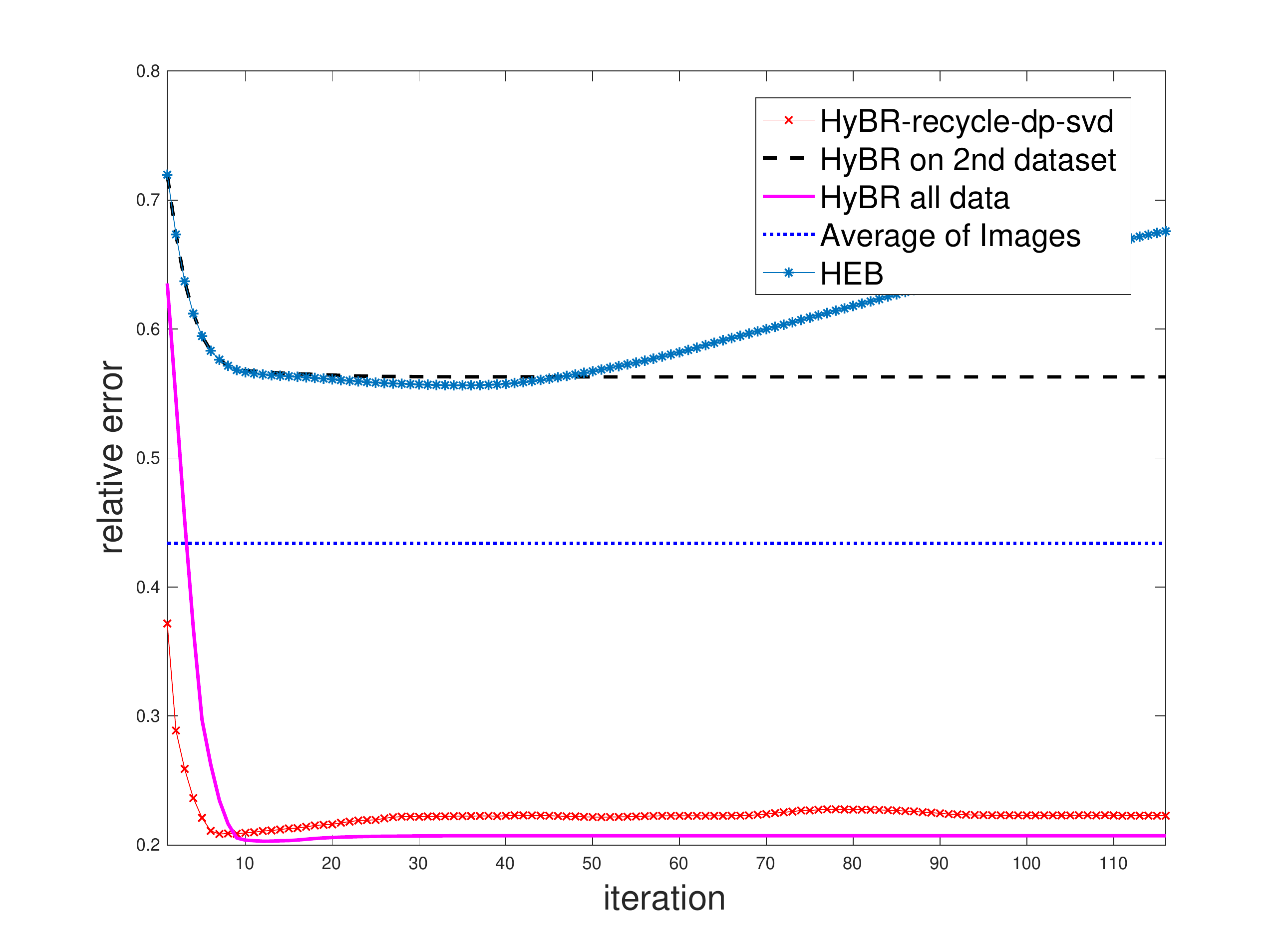}
  \caption{Streaming tomography example, case 1: Relative reconstruction error norms.}
  \label{fig:tomorel2}
\end{figure}

Image reconstructions with corresponding absolute error images are provided in \cref{fig:tomorecon2}. In terms of CPU time, HyBR-recycle-dp-svd took $79.85$ sec, HyBR with the 2nd dataset took $98.21$ sec and HyBR with the entire dataset took $181.64$ sec.
\begin{figure}[htbp]
\centering
  \includegraphics[width = \textwidth]{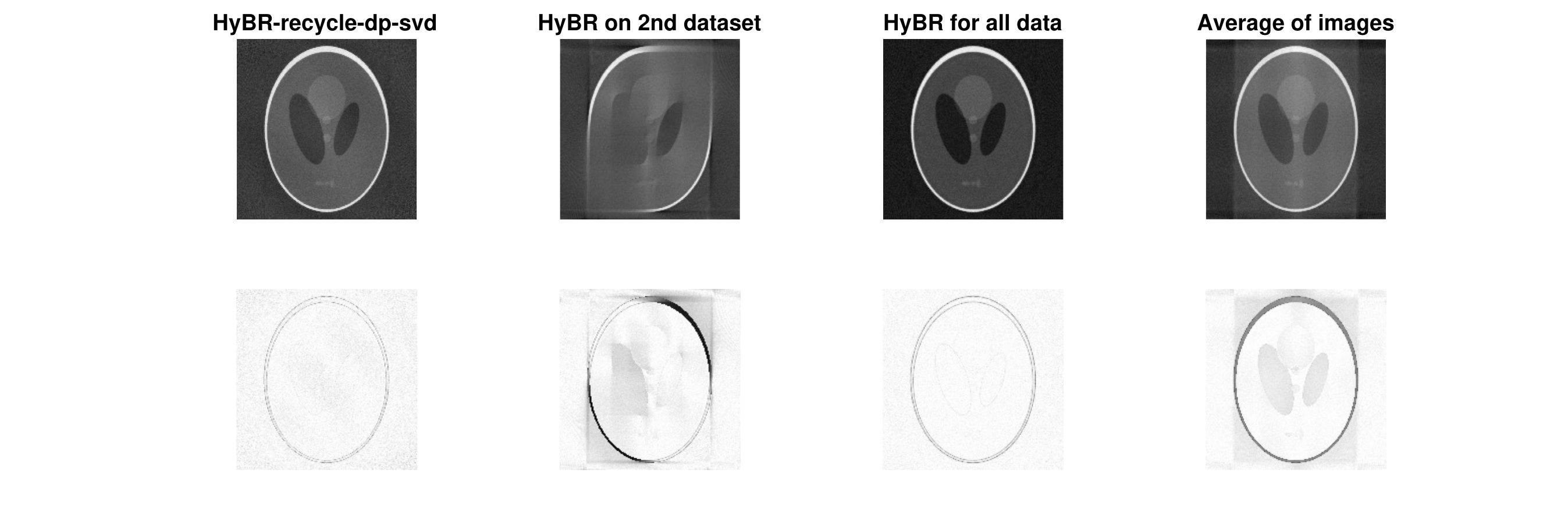}
  \caption{Streaming tomography example, case 1: reconstructions and error images (in inverted colormap).}
  \label{fig:tomorecon2}
\end{figure}

{\bf Case 2:} Next we assume that data is being streamed such that we have four reconstruction problems corresponding to $45$ equally spaced projection angles in $0^{\circ} - 44^{\circ}, 45^{\circ} - 90^{\circ}, 91^{\circ} - 135^{\circ},$ and $136^{\circ} - 179^{\circ}$ respectively. In terms of dimensions, $\bfA_1,\bfA_2,\bfA_3,\bfA_4 \in \bbR^{45\cdot1448 \times 1024^2}$ and $\bfb_1, \bfb_2, \bfb_3, \bfb_4 \in \bbR^{45\cdot1448}$. The noise level for each observed image is $0.02,$ which means that $\frac{\norm[2]{\bfepsilon_i}}{\norm[2]{\bfA_i \bfx_\true}} = 0.02$ for $i=1,2,3,4$.
The observed sinograms are provided in \cref{fig:tomo_ex2}. This time, we limit the storage of the solution basis vectors (each still $1,048,576 \times 1$) to be $15$. For the stopping criteria for compression, we set the maximum number of saved vectors after each compression to be 5 and the compression tolerance to be $\varepsilon_{tol} = 10^{-6}$. 
For the first reconstruction problem, we run standard HyBR for $m=15$ iterations.
\begin{figure}[htbp]
\centering
  \begin{tabular}{cccc}
  \includegraphics[width = .25\textwidth]{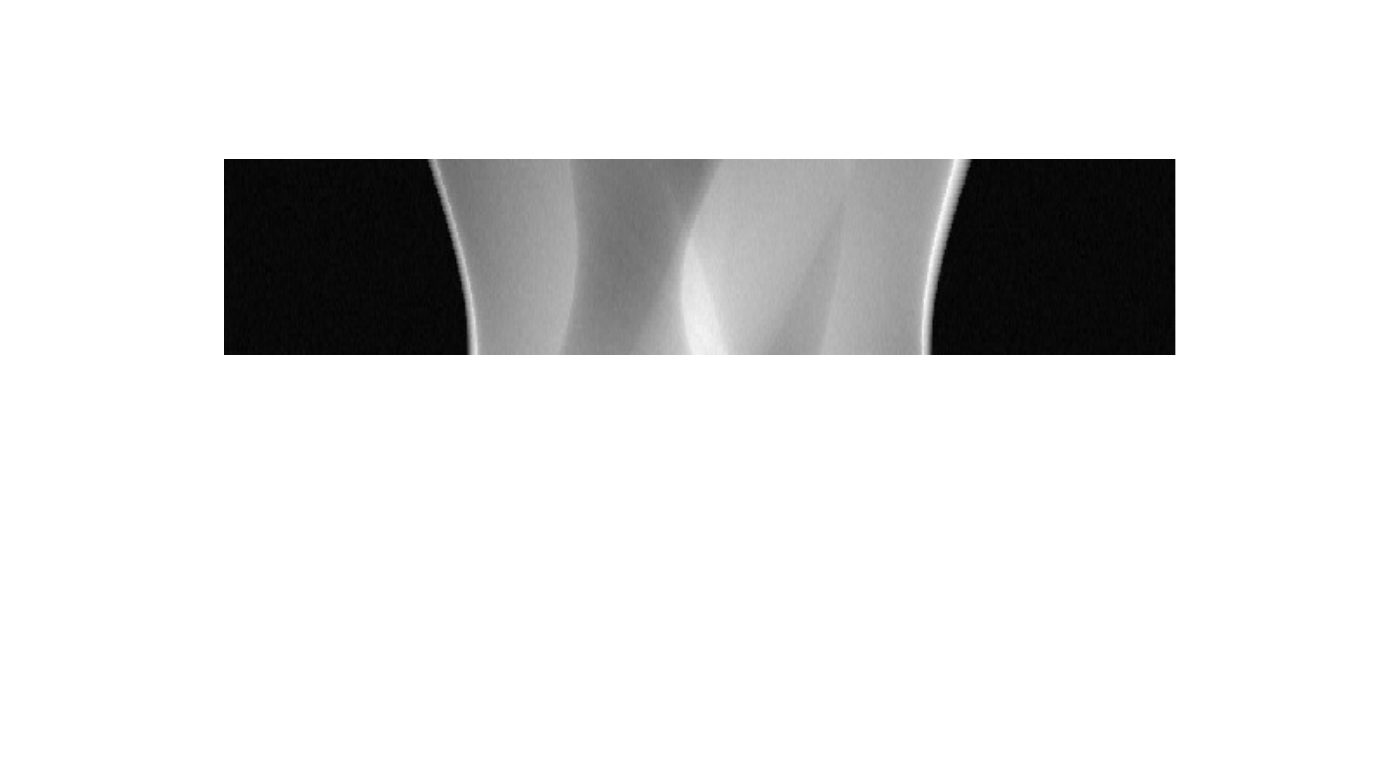} &
  \includegraphics[width = .2\textwidth]{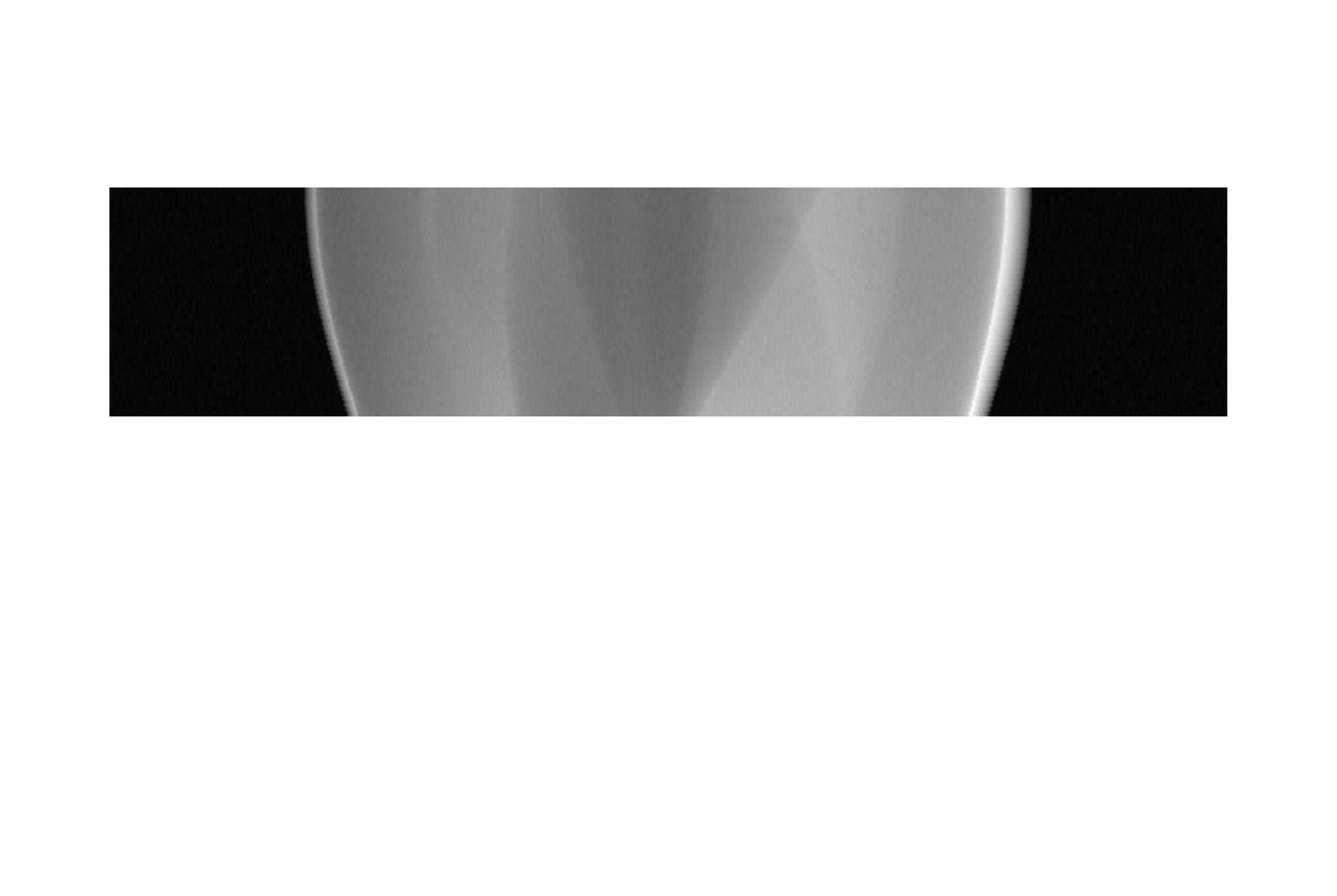} &
  \includegraphics[width = .2\textwidth]{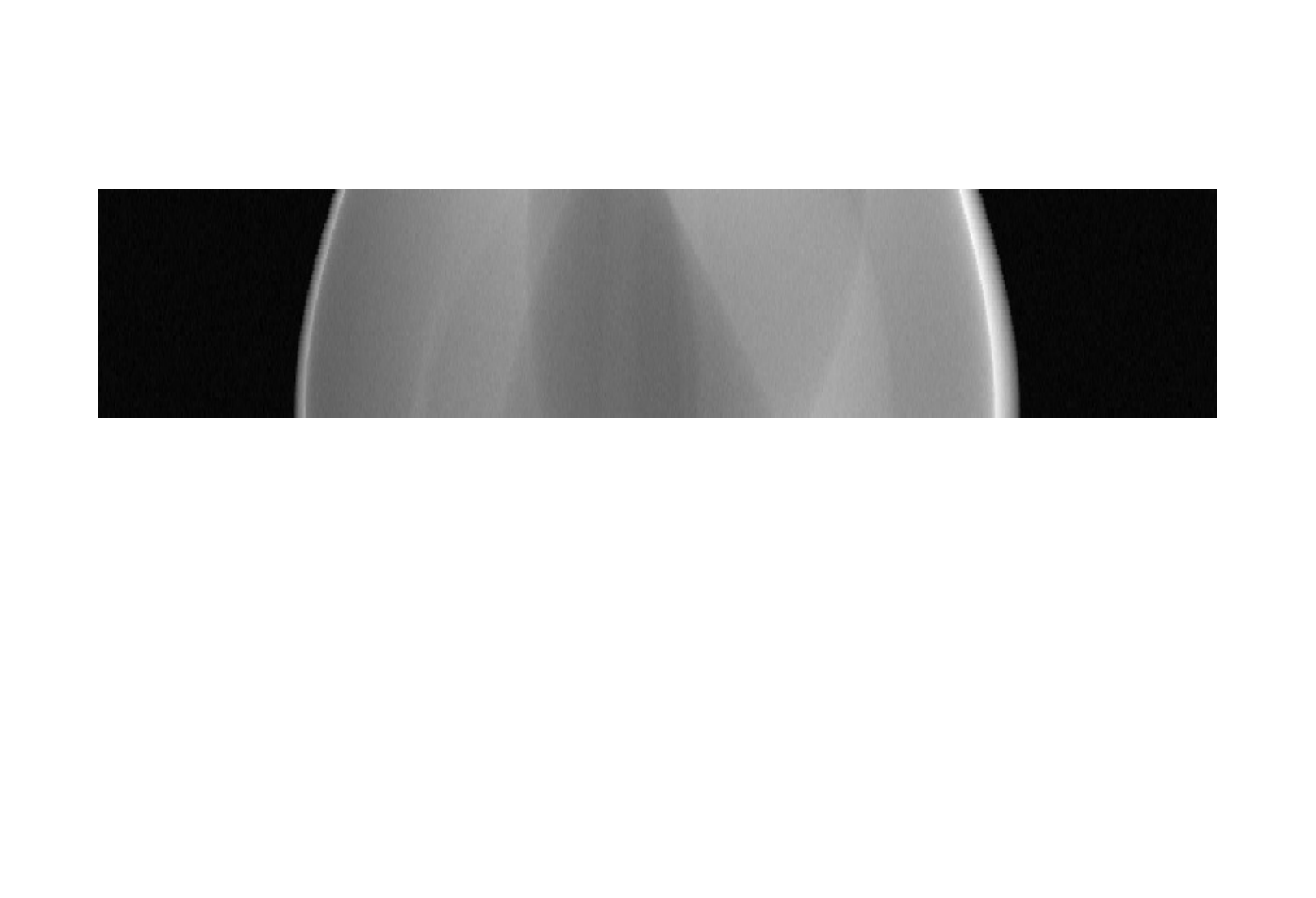} &
  \includegraphics[width = .2\textwidth]{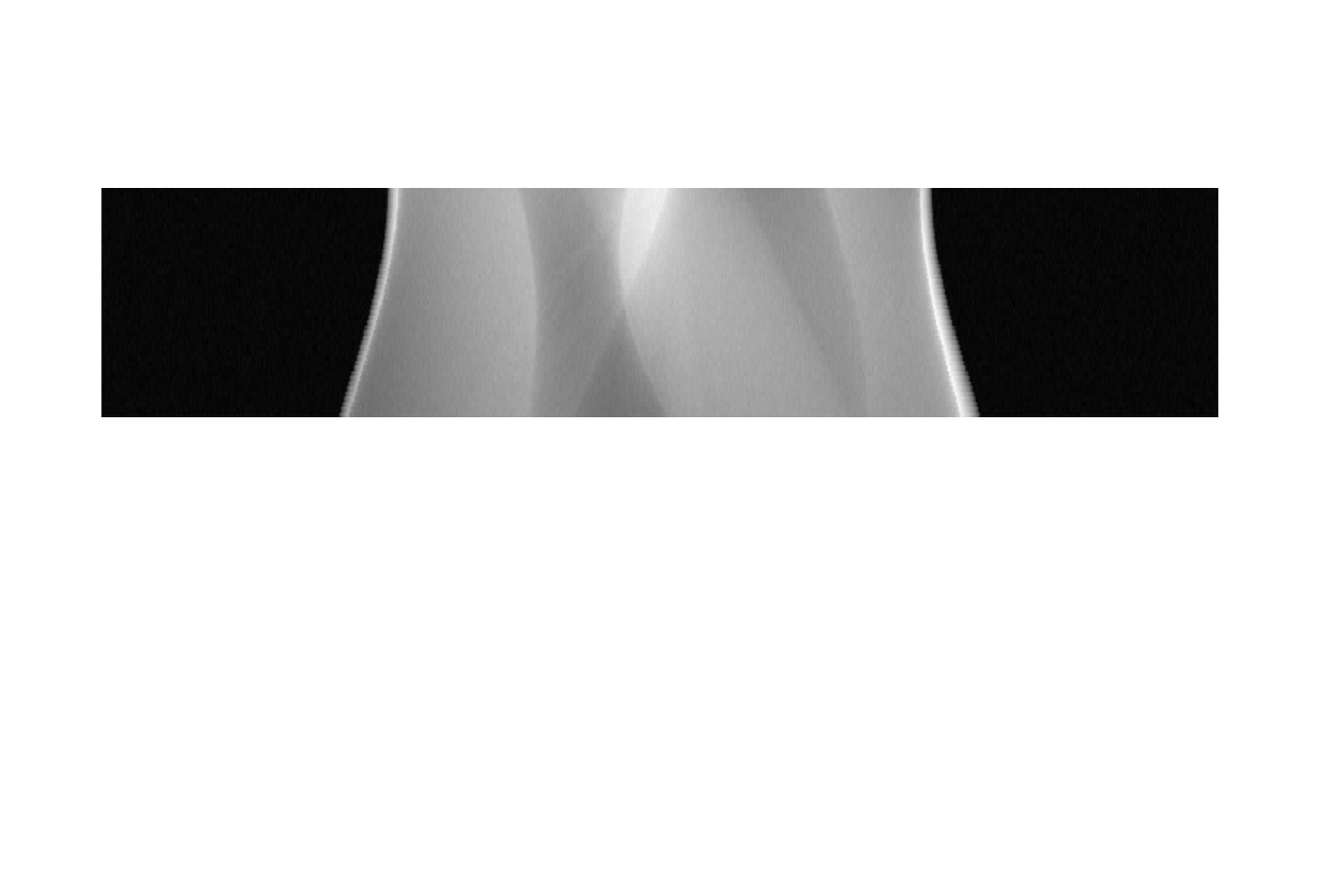} \\
(a) $\bfb_1$: $0^{\circ}-44^{\circ}$ & (b) $\bfb_2$: $45^{\circ}-89^{\circ}$ & (c) $\bfb_3$: $90^{\circ}-134^{\circ}$& (d) $\bfb_4$: $135^{\circ}-179^{\circ}$
\end{tabular}
  \caption{Streaming tomography example, case 2. The true image is provided in Figure \ref{fig:tomo_ex1} (a), and the four observed sinogram images are provided here.}
  \label{fig:tomo_ex2}
\end{figure}

A plot of the relative reconstruction error norms per iteration is provided in \cref{fig:tomorel3}.  For the HyBR-recycle-dp-svd method,
we provide relative reconstruction error norms from the second to the last reconstruction problem (that is, the $1st - 14th$ iterations correspond to the first problem with standard HyBR, the $15th - 31st$ iterations correspond to the second problem, the $32nd - 48th$ iterations correspond to the third problem, and the $49th - 110th$ iterations correspond to the fourth problem). 
\begin{figure}[htbp]
\centering
  \includegraphics[width = 0.85\textwidth]{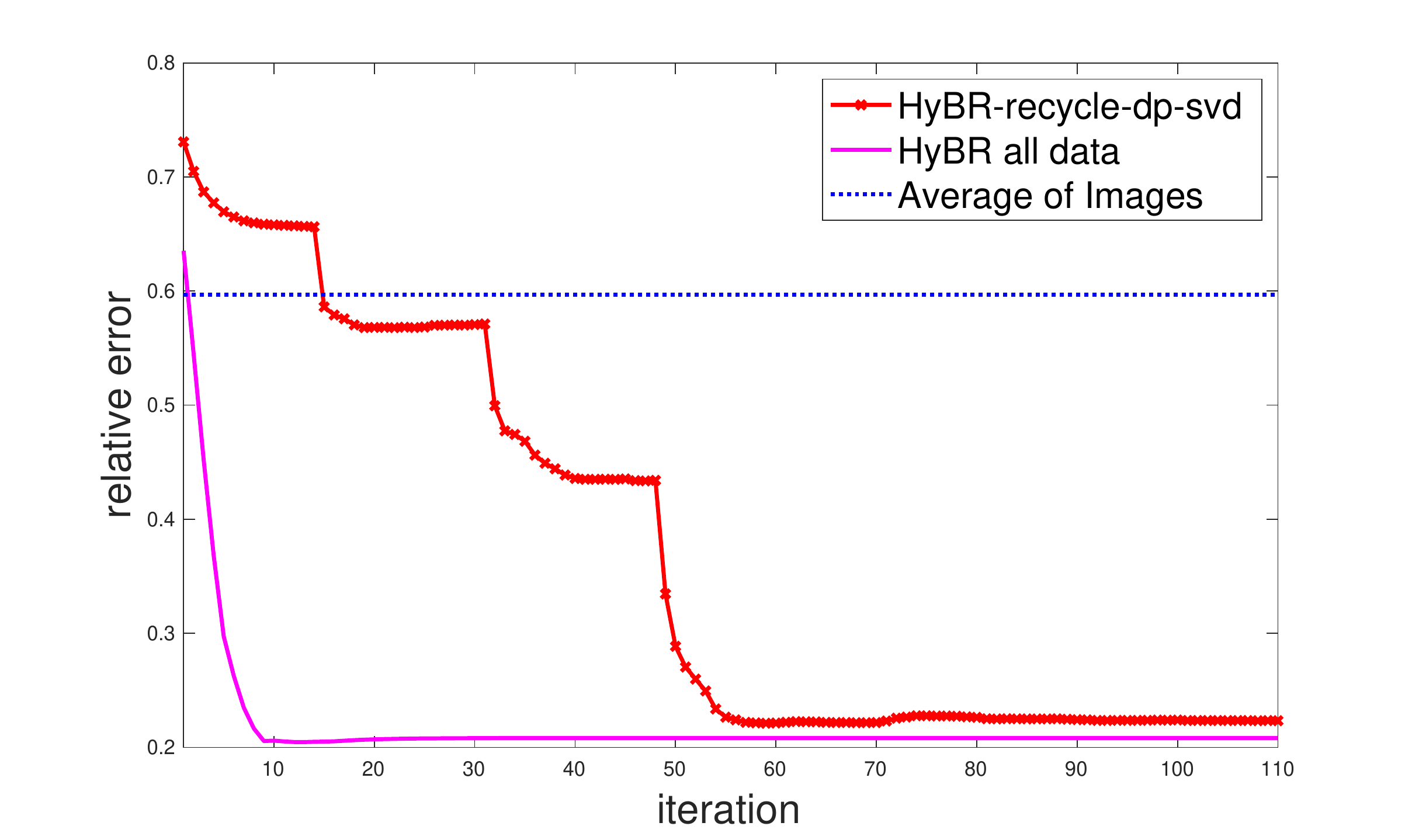}
  \caption{Streaming tomography example, case 2: Relative reconstruction error norms.}
  \label{fig:tomorel3}
\end{figure}
Image reconstructions with absolute error images are provided in \cref{fig:tomorecon2}. In terms of overall CPU time, HyBR-recycle took $43.10$ sec, HyBR with one dataset took $96.72$ sec (i.e., the time to compute an average of solutions), and HyBR with the entire dataset took $226.70$ sec. 

We observe that HyBR-recycle produces reconstructions with relative reconstruction error norms that are much smaller than the average of images. HyBR with all of the data is the most accurate approach, as expected, but it is more costly.  Furthermore, for large-scale \textit{sequential} problems where it is not desirable to wait until all data have been collected to perform reconstruction, HyBR-recycle provides an efficient approach to compute regularized solutions with comparable accuracy.
\begin{figure}[htbp]
\centering
  \includegraphics[width = \textwidth]{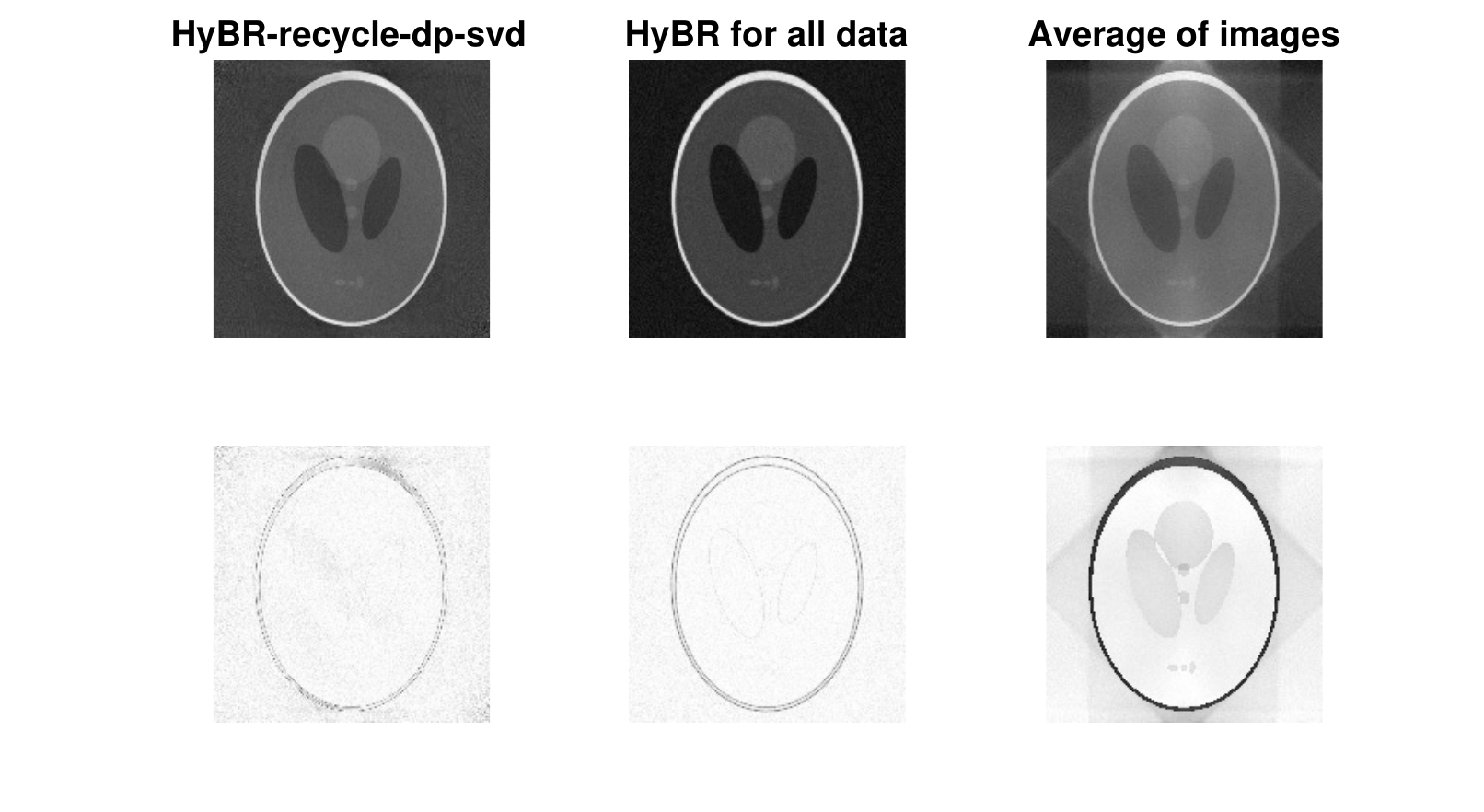}
  \caption{Streaming tomography example, case 2: reconstructions and error images (in inverted colormap).}
  \label{fig:tomorecon3}
\end{figure}
\begin{figure}[htbp]
\centering
  \includegraphics[width = \textwidth]{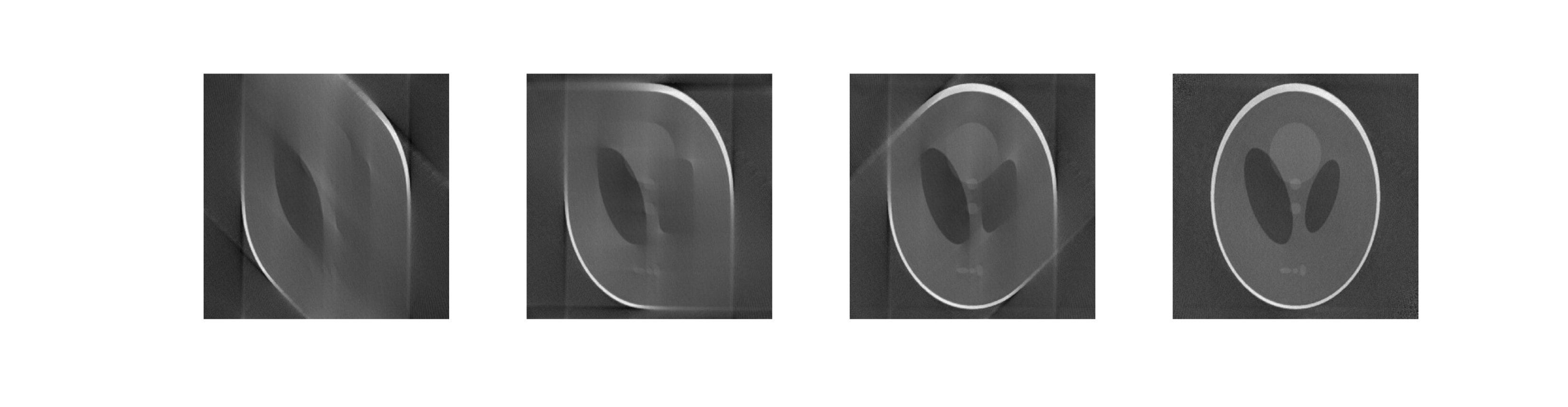}
  \caption{Streaming tomography example, case 2: reconstructions from the HyBR-recycle-dp-svd method at iterations $14,$ $31$, $48$, and $110$ respectively.}
  \label{fig:tomoreconrecycle3}
\end{figure}

\subsubsection{Tomographic reconstruction of a walnut}
\label{sec:realdata}
To test the practicality of the hybrid projection methods with recycling, we present reconstruction results from actual tomographic x-ray projection data from a walnut \cite{hamalainen2015tomographic}. This example consists of four reconstruction problems, where the projection angles for each problem are slightly modified. The need to solve multiple problems with modified projection angles arises in various scenarios including optimal experimental design frameworks \cite{ruthotto2018optimal} and optimization to correct for uncertain angles \cite{riis2019new}. We investigate the use of hybrid projection methods with recycling to re-use the solution and solution basis vectors acquired from one reconstruction to efficiently solve another reconstruction problem with modified angles. Since the data for this example are taken from real experiments, the true solution is not available.

We are given a set of 120 fan-beam projections taken at an angular step of three degrees. The number of rays per projection is 328. The first system corresponds to $30$ equally-spaced projection angles between $3^{\circ}$ and $351^{\circ}$ degrees, which gives $\bfA_1 \in \mathbb{R}^{30 \cdot 328 \times 328^2}$ and $\bfb_1 \in \mathbb{R}^{30 \cdot 328 }$.  The second system is generated using $30$ equally spaced projection angles between $6^{\circ}$ and $354^{\circ}$ degrees, which gives $\bfA_2 \in \mathbb{R}^{30 \cdot 328 \times 328^2}$ and $\bfb_2 \in \mathbb{R}^{30 \cdot 328 }$. The third system is generated using $30$ equally spaced projection angles between $9^{\circ}$ and $357^{\circ}$ degrees, which gives $\bfA_3 \in \mathbb{R}^{30 \cdot 328 \times 328^2}$ and $\bfb_3 \in \mathbb{R}^{30 \cdot 328 }$. The fourth system is generated using $30$ equally spaced projection angles between $12^{\circ}$ and $360^{\circ}$ degrees, which gives $\bfA_4 \in \mathbb{R}^{30 \cdot 328 \times 328^2}$ and $\bfb_4 \in \mathbb{R}^{30 \cdot 328 }$. 

For HyBR-recycle, we initialize with
$m=100$ iterations of the standard Golub-Kahan bidiagonalization with $\bfA_1$ and $\bfb_1$ to get $\bfx_{100}$, and we compress the basis vectors $\bfV_{100}$ to get $\bfW_{90}$. Then $\bfW_{91} = \begin{bmatrix}\bfW_{90} & \widecheck{\bfx}^{(1)}\end{bmatrix}$, 
where $\widecheck{\bfx}^{(1)} = (\bfx^{(1)} - \bfW_{90}\bfW_{90}^{\top}\bfx^{(1)})/\norm[2]{\bfx^{(1)} - \bfW_{90}\bfW_{90}^{\top}\bfx^{(1)}}$ and $\bfx^{(1)} = \bfx_{100}$. 
Given the initial set of basis vectors in $\bfW_{91}$, we use HyBR-recycle with the four different compression techniques described in \cref{sec:compression_approaches}.  For all of the considered methods for this problem, we allow storage for a maximum of $100$ solution vectors. For HyBR-recycle, the maximum number of vectors to save at compression is $90$ with the compression tolerance being $\varepsilon_{tol} = 10^{-6}$, and we allow two cycles of HyBR-recycle for each dataset. 

We compare these results for HyBR-recycle to HyBR with the fourth dataset, HyBR with all data, and the average of images obtained from four HyBR reconstructions.
The reconstructions of the standard approaches are obtained after $100$ iterations.  For all of these experiments, GCV is used to select the regularization parameter.  From the image reconstructions provided in \cref{fig:tomorecon4}, we observe that the lack of data for HyBR with the fourth dataset results in artifacts, whereas the average of images is quite blurry. We only provide the HyBR-recycle reconstruction using solution-oriented compression, but we remark that similar results were observed for all compression approaches. Let the HyBR for all data solution be denoted as $\bfx_{all}$, and define the relative difference as $\norm[2]{\bfx - \bfx_{all}}/\norm[2]{\bfx_{all}}$. These values are provided in \cref{table5} for the various reconstructions. 
In summary, the HyBR-recycle reconstructions contain some noise, but are overall sharper than taking an average of image reconstructions (this is clear especially around the edges), and do not suffer from artifacts from limited data.

\begin{figure}[htbp]
\hspace{-1in}
  \includegraphics[width = 1.3\textwidth]{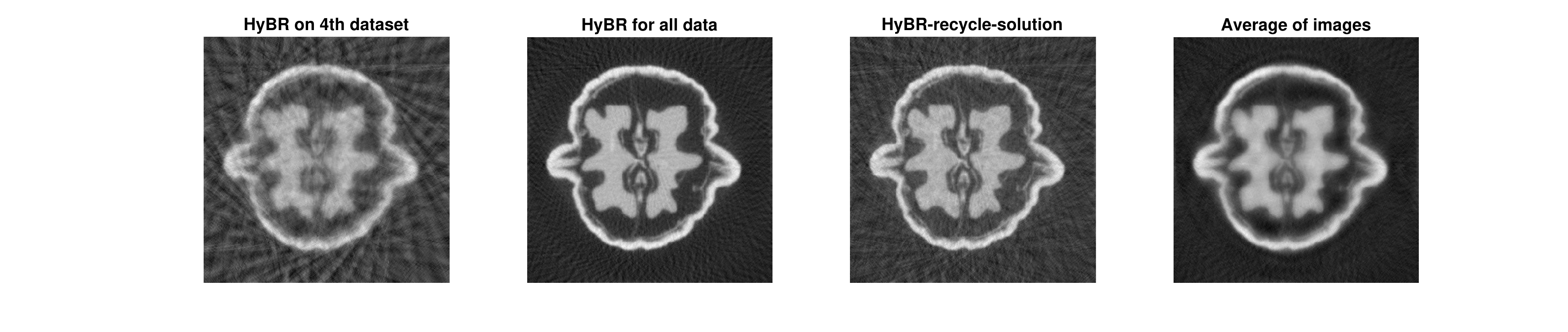}
  \caption{Tomography walnut example: Image reconstructions.}
  \label{fig:tomorecon4}
\end{figure}

\begin{table}[h!]
\centering
 \begin{tabular}{|c| c| c| c|}
 \hline
 HyBR with 4th dataset& 0.3102 & HyBR-recycle-gcv-tsvd   & 0.1814 \\
 \hline
Average of images & 0.2679 &
 HyBR-recycle-gcv-solution & 0.1841 \\
 \hline
 HyBR-recycle-gcv-rbd  & 0.1732
 & HyBR-recycle-gcv-sparse  & 0.1832 \\
 \hline
 \end{tabular}
 \caption{Tomography walnut example: Relative differences computed as $\norm[2]{\bfx - \bfx_{all}}/\norm[2]{\bfx_{all}}$ where $\bfx$ represents the numerical solution computed by HyBR, Average of images, and HyBR-recycle with different compression techniques.  All results use GCV for selecting the regularization parameter.}
 \label{table5}
\end{table}

\section{Conclusions}
\label{sec:conclusions}
In this paper, we have described Golub-Kahan-based hybrid projection methods with recycling that use compression and recycling to overcome potential memory limitations.  We described a variety of problems that can be solved using these methods.
For example, we can solve very large problems where the number of basis vectors becomes too large for memory storage.  These methods can be used to efficiently solve a sequence of regularized problems (e.g., changing regularization terms or nonlinear solvers) and problems with streaming data.  We emphasize that the general approach can also be used in an iterative fashion to improve on existing solutions. The main computational benefits include improved regularized solutions, reduced memory requirements, and automatic selection of the regularization parameter.
Theoretical results show connections between projected problems and relationships between regularized solutions, and numerical results demonstrate that our approach can efficiently and accurately solve various inverse problems from image processing.

\bibliographystyle{abbrv}
\bibliography{inversebiblio}

\begin{thebibliography}{10}

\bibitem{Ahuja.Recyc-BiCGStab-MOD.2015}
K.~Ahuja, P.~Benner, E.~de~Sturler, and L.~Feng.
\newblock Recycling {B}i{CGSTAB} with an application to parametric model order
  reduction.
\newblock {\em SIAM J. Sci. Comput.}, 37(5):S429--S446, 2015.

\bibitem{baglama2005augmented}
J.~Baglama and L.~Reichel.
\newblock Augmented implicitly restarted {Lanczos} bidiagonalization methods.
\newblock {\em SIAM Journal on Scientific Computing}, 27(1):19--42, 2005.

\bibitem{baglama2013augmented}
J.~Baglama, L.~Reichel, and D.~Richmond.
\newblock An augmented {LSQR} method.
\newblock {\em Numerical Algorithms}, 64(2):263--293, 2013.

\bibitem{Bjo88}
A.~Bj{\"o}rck.
\newblock A bidiagonalization algorithm for solving large and sparse ill-posed
  systems of linear equations.
\newblock {\em BIT}, 28:659--670, 1988.

\bibitem{calvetti2003enriched}
D.~Calvetti, L.~Reichel, and A.~Shuibi.
\newblock Enriched {K}rylov subspace methods for ill-posed problems.
\newblock {\em Linear algebra and its applications}, 362:257--273, 2003.

\bibitem{chen2015reduced}
Y.~Chen.
\newblock Reduced basis decomposition: a certified and fast lossy data
  compression algorithm.
\newblock {\em Computers \& Mathematics with Applications}, 70(10):2566--2574,
  2015.

\bibitem{ChNaOLe08}
J.~Chung, J.~G. Nagy, and D.~P. O'Leary.
\newblock A weighted {GCV} method for {L}anczos hybrid regularization.
\newblock {\em Elec.~Trans.~Numer.~Anal.}, 28:149--167, 2008.

\bibitem{feng2009parametric}
L.~Feng, P.~Benner, and J.~G. Korvink.
\newblock Parametric model order reduction accelerated by subspace recycling.
\newblock In {\em Proceedings of the 48h IEEE Conference on Decision and
  Control (CDC) held jointly with 2009 28th Chinese Control Conference}, pages
  4328--4333. IEEE, 2009.

\bibitem{gazzola2017ir}
S.~Gazzola, P.~C. Hansen, and J.~G. Nagy.
\newblock I{R} tools: a {MATLAB} package of iterative regularization methods
  and large-scale test problems.
\newblock {\em Numerical Algorithms}, 81(3):773--811, 2019.

\bibitem{GoKa65}
G.~Golub and W.~Kahan.
\newblock Calculating the singular values and pseudoinverse of a matrix.
\newblock {\em SIAM J.~Numer.~Anal.}, 2:205--224, 1965.

\bibitem{hamalainen2015tomographic}
K.~H{\"a}m{\"a}l{\"a}inen, L.~Harhanen, A.~Kallonen, A.~Kujanp{\"a}{\"a},
  E.~Niemi, and S.~Siltanen.
\newblock Tomographic x-ray data of a walnut.
\newblock {\em arXiv preprint arXiv:1502.04064}, 2015.

\bibitem{hansen2010discrete}
P.~C. Hansen.
\newblock {\em Discrete Inverse Problems: Insight and Algorithms}.
\newblock SIAM, Philadelphia, PA, 2010.

\bibitem{hansen2019hybrid}
P.~C. Hansen, Y.~Dong, and K.~Abe.
\newblock Hybrid enriched bidiagonalization for discrete ill-posed problems.
\newblock {\em Numerical Linear Algebra with Applications}, 26(3):e2230, 2019.

\bibitem{hansen2018air}
P.~C. Hansen and J.~S. Jorgensen.
\newblock {AIR Tools II}: algebraic iterative reconstruction methods, improved
  implementation.
\newblock {\em Numerical Algorithms}, 79(1):107--137, 2018.

\bibitem{hochstenbach2010subspace}
M.~E. Hochstenbach and L.~Reichel.
\newblock Subspace-restricted singular value decompositions for linear discrete
  ill-posed problems.
\newblock {\em Journal of Computational and Applied Mathematics},
  235(4):1053--1064, 2010.

\bibitem{jin2007parallel}
C.~Jin, X.-C. Cai, and C.~Li.
\newblock Parallel domain decomposition methods for stochastic elliptic
  equations.
\newblock {\em SIAM Journal on Scientific Computing}, 29(5):2096--2114, 2007.

\bibitem{keuchel2016combination}
S.~Keuchel, J.~Biermann, and O.~von Estorff.
\newblock A combination of the fast multipole boundary element method and
  {K}rylov subspace recycling solvers.
\newblock {\em Engineering Analysis with Boundary Elements}, 65:136--146, 2016.

\bibitem{KilmerdeSturler2006}
M.~Kilmer and E.~de~Sturler.
\newblock {Recycling subspace information for diffuse optical tomography}.
\newblock {\em SIAM J. Sci. Comput.}, 27(6):2140--2166, 2006.

\bibitem{KiOLe01}
M.~E. Kilmer and D.~P. O'Leary.
\newblock Choosing regularization parameters in iterative methods for ill-posed
  problems.
\newblock {\em SIAM J.~Matrix Anal.~Appl.}, 22(4):1204--1221, 2001.

\bibitem{Recyc-Imp-Tom.2010}
L.~A.~M. Mello, E.~de~Sturler, G.~H. Paulino, and E.~C.~N. Silva.
\newblock Recycling {K}rylov subspaces for efficient large-scale electrical
  impedance tomography.
\newblock {\em Comput. Methods Appl. Mech. Engrg.}, 199(49-52):3101--3110,
  2010.

\bibitem{nagy2004iterative}
J.~G. Nagy, K.~Palmer, and L.~Perrone.
\newblock Iterative methods for image deblurring: a matlab object-oriented
  approach.
\newblock {\em Numerical Algorithms}, 36(1):73--93, 2004.

\bibitem{OlSi81}
D.~P. O'Leary and J.~A. Simmons.
\newblock A bidiagonalization-regularization procedure for large scale
  discretizations of ill-posed problems.
\newblock {\em SIAM J.~Sci.~Comput.}, 2(4):474--489, 1981.

\bibitem{parkinson2017machine}
D.~Y. Parkinson, D.~M. Pelt, T.~Perciano, D.~Ushizima, H.~Krishnan, H.~S.
  Barnard, A.~A. MacDowell, and J.~Sethian.
\newblock Machine learning for micro-tomography.
\newblock In {\em Developments in X-Ray Tomography XI}, volume 10391, page
  103910J. International Society for Optics and Photonics, 2017.

\bibitem{parks2006recycling}
M.~L. Parks, E.~de~Sturler, G.~Mackey, D.~D. Johnson, and S.~Maiti.
\newblock Recycling {K}rylov subspaces for sequences of linear systems.
\newblock {\em SIAM Journal on Scientific Computing}, 28(5):1651--1674, 2006.

\bibitem{renaut2010regularization}
R.~A. Renaut, I.~Hnetynkov{\'a}, and J.~Mead.
\newblock Regularization parameter estimation for large-scale {Tikhonov}
  regularization using a priori information.
\newblock {\em Computational Statistics \& Data Analysis}, 54(12):3430--3445,
  2010.

\bibitem{riis2019new}
N.~A.~B. Riis and Y.~Dong.
\newblock A new iterative method for ct reconstruction with uncertain view
  angles.
\newblock In {\em International Conference on Scale Space and Variational
  Methods in Computer Vision}, pages 156--167. Springer, 2019.

\bibitem{ruthotto2018optimal}
L.~Ruthotto, J.~Chung, and M.~Chung.
\newblock Optimal experimental design for inverse problems with state
  constraints.
\newblock {\em SIAM Journal on Scientific Computing}, 40(4):B1080--B1100, 2018.

\bibitem{slagel2018sampled}
J.~T. Slagel, J.~Chung, M.~Chung, D.~Kozak, and L.~Tenorio.
\newblock Sampled {T}ikhonov regularization for large linear inverse problems.
\newblock {\em Inverse Problems}, 2019.

\bibitem{S.2016}
K.~M. Soodhalter.
\newblock Block {K}rylov subspace recycling for shifted systems with unrelated
  right-hand sides.
\newblock {\em SIAM Journal on Scientific Computing}, 38(1):A302--A324, 2016.

\bibitem{soodhalter2014krylov}
K.~M. Soodhalter, D.~B. Szyld, and F.~Xue.
\newblock Krylov subspace recycling for sequences of shifted linear systems.
\newblock {\em Applied Numerical Mathematics}, 81:105--118, 2014.

\bibitem{Wang.TopOptRecyc.2007}
S.~Wang, E.~de~Sturler, and G.~H. Paulino.
\newblock Large-scale topology optimization using preconditioned {K}rylov
  subspace methods with recycling.
\newblock {\em International Journal for Numerical Methods in Engineering},
  69(12):2441--2468, 2007.

\bibitem{wright2009sparse}
S.~J. Wright, R.~D. Nowak, and M.~A. Figueiredo.
\newblock Sparse reconstruction by separable approximation.
\newblock {\em IEEE Transactions on Signal Processing}, 57(7):2479--2493, 2009.

\end{thebibliography}

\end{document}